\documentclass[10pt]{amsart}
\usepackage{amsmath}
\usepackage{amsfonts}
\usepackage{amssymb}
\usepackage{amsthm}
\usepackage{graphicx}
\usepackage{hyperref}

\usepackage{color}

\def \B {\hat{B}}

\def \d {\hat{d}}
\def \de {\partial}
\def \e {\varepsilon}
\def \G {\mathrm{G}}
\def \H {\mathcal{H}}
\def \HH {\textsc{{H}} }
\def \N {\mathbb{N}}
\def \O {\Omega}
\def \phi {\varphi}
\def \RN {\mathbb{R}^N}
\def \RNu {\mathbb{R}^{N+1}}
\def \R {\mathbb{R}}
\def \l {\lambda}
\def \L {\Lambda}

\def \Ockq {\Omega^c_{kq+i}(z_0)}
\def \Tq {T_{kq+i}}
\def \Tst {T^*_{kq+i}}
\def \P {\mathcal{P}(z,r)}
\def \div {{\rm{div}}}

\newtheorem{theorem}{Theorem}[section]
\newtheorem{lemma}[theorem]{Lemma}
\newtheorem{proposition}[theorem]{Proposition}
\newtheorem{corollary}[theorem]{Corollary}

\newtheorem{remark}[theorem]{Remark}

\theoremstyle{definition}

\numberwithin{equation}{section}

\begin{document}

\title[A Wiener test \`a la Landis for evolutive H\"ormander operators]{A Wiener test \`a la Landis \\for evolutive H\"ormander operators}

\author[G. Tralli]{Giulio Tralli}
\address{Dipartimento d'Ingegneria Civile e Ambientale (DICEA)\\ Universit\`a di Padova\\ Via Marzolo, 9 - 35131 Padova,  Italy.
         }
 \email[Corresponding author]{giulio.tralli@unipd.it}
\author[F. Uguzzoni]{Francesco Uguzzoni}
\address{Dipartimento di Matematica,
         Universit\`{a} degli Studi di Bologna,
         Piazza di Porta S. Donato, 5 - 40126 Bologna, Italy.        
         }
 \email{francesco.uguzzoni@unibo.it}


\date{}

\begin{abstract}
In this paper we prove a Wiener-type characterization of boundary regularity, in the spirit of a classical result by Landis, for a class of evolutive H\"ormander operators. We actually show the validity of our criterion for a larger class of degenerate-parabolic operators with a fundamental solution satisfying suitable two-sided Gaussian bounds. Our condition is expressed in terms of a series of balayages or, (as it turns out to be) equivalently, Riesz-potentials.
\end{abstract}
\maketitle

\section{Introduction}\label{intro}

We are interested in Wiener type criteria of regularity of boundary points for evolutive hypoelliptic operators. The case of the classical heat equation and of uniformly parabolic operators in divergence form has been settled respectively by Evans-Gariepy \cite{EG} and by Garofalo-Lanconelli \cite{GL} (see below for more detailed historical notes). As far as we know, there is no characterization results of Wiener type even for the general H\"ormander model operator
\begin{equation}\label{mod}
\sum_j{X^2_j-\de_t}.
\end{equation}
In such sub-Riemannian settings, the only Evans-Gariepy Wiener criterion is in fact due to Garofalo and Segala in \cite{GS} for the heat equation on the Heisenberg group (see also the recent work in \cite{Ro} dealing with the case of H-type groups). 
On the other hand, the papers \cite{Sc, KLT} deal with Wiener tests of Landis-type for the special class of Kolmogorov equations. In all these papers, the precise knowledge of the fundamental solution plays a crucial role. A different approach has been carried out in \cite{LTU, TUcv} for H\"ormander operators, but the necessary and the sufficient condition for the regularity are different.

In the present paper we prove a characterization result \`a la Wiener-Landis for a class of evolutive operators containing \eqref{mod}. Actually our class contains in particular the operators in the form
\begin{equation}\label{nondivmod}
\sum_{i,j=1}^p{a_{i,j}(z)X_iX_j}  +\sum_{j=1}^p{b_{j}(z)X_j}    -\partial_t,\quad\mbox{ for }z=(x,t)\in D\times]T_1,T_2[,
\end{equation}
where $D\subset\RN$ is bounded and open, the smooth vector fields $\{X_1,\ldots, X_p\}$ satisfy the H\"ormander rank condition in a bounded open set $D_0\supset \supset D $, $a_{i,j}, b_{j}$ are smooth functions in $D_0\times]T_1,T_2[$, and the matrix $(a_{i,j}(\cdot))_{i,j}$ is symmetric and uniformly positive definite. H\"{o}rmander-type operators arise in many theoretical and applied settings sharing a sub-Riemannian underlying geometry, for instance in mathematical models for finance, control theory, geometric measure theory, pseudohermitian and CR geometry.

\noindent Relatively to operators in \eqref{nondivmod}, our main result (Theorem \ref{iff} below) reads as follows:
\begin{center} if $\Omega$ is a bounded open set which is compactly contained in $D\times]T_1,T_2[$, and $z_0\in\partial\Omega$, then
\begin{equation}\label{iffHo}
z_0 \mbox{ is } \H\mbox{-regular for }\de \Omega \qquad\Longleftrightarrow\qquad\sum_{k=1}^{\infty}{V_{\Omega^c_k(z_0)}(z_0)}=+\infty.
\end{equation}
\end{center}
Here, $V_{\Omega^c_k(z_0)}$ denotes the balayage of some compact sets $\Omega^c_k(z_0)$ involving suitable level sets of the fundamental solution of the operator $\H$ under consideration (see the following sections for the precise definitions).

Even for the heat operator, Wiener-type characterizations have a long history. To the best of our knowledge, the first attempt in this direction is due to Pini in \cite{Pi} where he proved a sufficient condition in the $1$-dimensional case for particular open sets. Then, in \cite{La} Landis proved a characterization for the regularity in terms of a suitable series of caloric potentials. Concerning the proper analogue of the classical Wiener criterion for the heat equation, Lanconelli proved in \cite{L73} the necessary condition for the regularity and, finally, Evans and Gariepy proved the full characterization in \cite{EG}.\\
It is well-known that all the elliptic operators share the same regular points with the Laplacian, whereas Petrowski showed in \cite{Pe} explicit counterexamples of this fact even for constant coefficients parabolic operators. This feature makes more interesting the study of the variable coefficients case. Several necessary and sufficient conditions have been investigated for classical parabolic operators both in divergence and non divergence form, also with different degree of regularity for the coefficients (see, e.g., \cite{La, Nov, L77} and references therein). The Evans-Gariepy Wiener test was extended to parabolic operators in divergence form with smooth variable coefficients by Garofalo and Lanconelli in \cite{GL}, and with $C^1$-Dini continuous coefficients by Fabes-Garofalo-Lanconelli in \cite{FGL}. We also mention \cite{KL, KKKP, BBG, AKN} (and references therein) for some recent developments in quasilinear parabolic settings.

We now turn back to the sub-Riemannian setting in order to put our result in perspective with respect to the state of the art already mentioned. In \cite{LTU, TUcv} we found necessary and sufficient conditions (different from each other) which are uniform in the class of operators \eqref{nondivmod}. Such conditions were expressed in terms of a series of capacities of compact sets involving only the underlying metric, whereas in the true characterization \eqref{iffHo} of the present paper we express the condition with balayages of super-level sets of the fundamental solution $\Gamma(\cdot,\cdot)$ of each operator $\H$ in the class. To do this we follow an approach which is more in the spirit of \cite{KLT}. One of the thorny issues of this strategy is to choose appropriately subregions of $\Omega^c_k(z_0)$ where we can estimate uniformly the ratio $\frac{\Gamma(z,\zeta)}{\Gamma(z_0,\zeta)}$. In contrast with the homogeneous Kolmogorov case in \cite{KLT}, we have to face additional difficulties such as the lack of an explicit knowledge of the fundamental solution and the lack of good scaling properties for the operators. Another problem we have faced in pursuing this strategy is the identification of the balayages with their Riesz representatives. Indeed, while the almost everywhere identification is quite straightforward, \emph{everywhere} identification seems to be a delicate point. One can approach such a Riesz representation theorem by making use of mean value formulas: for operators as in \eqref{nondivmod} the kernel in the mean value formulas may change sign and a careful analysis is in order.

It turns out that in our approach we can use essentially only two-sided Gaussian estimates for $\Gamma$ with respect to a well-behaved distance. For this reason we decided to present the results for a more general class of diffusion operators by using an axiomatic approach in the spirit of \cite{LU, LTU}. In the following subsection, we proceed by fixing precisely the class of operators under consideration.

\subsection{Assumption and main results}\label{subsec11}

Let us consider the following linear second order Partial Differential Operators
\begin{equation}\label{IHG}
\H=\sum_{i,j=1}^N q_{i,j}(z)\de^2_{x_i,x_j} +\sum_{k=1}^N q_{k}(z)\de_{x_k}
-\partial_t ,
 \end{equation}
in the strip of $\R^{N+1}$
$$ S=\{ z=(x,t)\, :\, x\in\RN,\ T_1<t<T_2\},\quad -\infty\leq T_1<T_2\leq \infty.$$
We assume the coefficients $q_{i,j}=q_{j,i}, q_k$
of class $C^\infty$, and the characteristic form
$$q_\H(z,\xi)=\sum_{i,j=1}^N q_{i,j}(z)\xi_i\xi_j,\quad \xi=(\xi_1,\ldots,\xi_N)\in\RN,$$
nonnegative definite and not totally degenerate, i.e.,
$q_\H(z,\cdot)\geq 0,\ q_\H(z,\cdot)\not\equiv 0$ for every $z\in S$. We also assume the {\em hypoellipticity} of $\H$ and of its adjoint $\H^*$, and the existence of a global {\em fundamental solution}
$$(z,\zeta )\mapsto \Gamma (z,\zeta )$$
smooth out of the diagonal of $S\times S$ satisfying the following:
\begin{enumerate}
\item[(i)] $\Gamma (\cdot,\zeta )\in  L^1_{\rm {loc}}(S)$ and $\H(\Gamma (\cdot,\zeta ))=-\delta _\zeta$, the Dirac measure at
 $\{\zeta\}$, for every $\zeta \in S$; $\Gamma (z,\cdot )\in  L^1_{\rm {loc}}(S)$ and $\H^*(\Gamma (z,\cdot ))=-\delta _z$ for every $z \in S$;
\item[(ii)] for every compactly supported continuous function $\phi$ on $\RN$ and for every $x_0\in\RN$, we have
\begin{equation}\label{Hbissa}
\int_{\RN}{\Gamma(x,t,\xi,\tau)\,\varphi(\xi)\,{\rm d}\xi}\to\phi(x_0)
\end{equation}
as $x\to x_0$, $t\searrow\tau\in]T_1,T_2[$ and also as $x\to x_0$, $\tau\nearrow t\in]T_1,T_2[$;
\item[(iii)] there exists a distance $d$ in $\RN$ verifying the properties (\hyperref[diuno]{D1})--(\hyperref[ditre]{D3}) below, and there exist constants $0<a_0\leq b_0$ and $\Lambda\geq 1$ such that the following Gaussian estimates hold
\begin{equation}\label{bounds}
\frac{1}{\Lambda} \G_{b_0}(z,\zeta)\leq \Gamma(z,\zeta)\leq\Lambda\G_{a_0}(z,\zeta),\quad \forall z,\zeta\in S.
\end{equation}
Hereafter, we denote by $\G_{a}$ the function
$$\G_a(z,\zeta)=\G_a(x,t,\xi,\tau)=
\begin{cases}
0 & \text{ if }t\le\tau,\\
\frac{1}{|B_d(x,\sqrt{t-\tau})|} \exp\left(-a \frac{d(x,\xi)^2}{t-\tau}\right) &  \text{ if }t>\tau.
\end{cases}
$$
\end{enumerate}

\begin{remark}\label{r12}
In particular, condition (ii) holds true if $\int_{\RN}{\Gamma(x,t,\xi,\tau)\,{\rm d}\xi}= 1$ (for any fixed $x$ and $t>\tau$) and (iii) is satisfied (see Remark \ref{eqass} below).
\end{remark} 

We fix here the notations we have just used. If $A\subseteq\RN$ ($A\subseteq\mathbb{R}^{N+1}$), $|A|$ denotes the $N$-dimensional ($(N+1)$-dimensional) Lebesgue measure of $A$. Moreover, we denote the $d$-ball of center $x$ and radius $r>0$ as
$$B_d(x,r)=B(x,r)=\{y\in\RN\,:\,d(x,y)<r\}.$$ 
Finally, we shall make the following assumptions on the metric space ($\R^N, d$):
\begin{enumerate}
\item[(D1)]\label{diuno} The $d$-topology is the Euclidean topology. Moreover $(\RN,d)$ is complete and, for every fixed $x\in\RN$, $d(x,\xi)\to\infty$ if (and only if) $\xi\to\infty$ with respect to the usual Euclidean norm.
\item[(D2)]\label{didue} $(\RN,d)$ is a {\em doubling metric space} w.r.t. the Lebesgue measure, i.e. there exists a constant $c_d>1$ such that
$$ |B(x,2r)|\le c_d |B(x,r)|, \quad \forall x\in\RN,\ \forall r>0.$$
We will always denote by $Q=\log_2{c_d}$ the relative homogeneous dimension.
\item[(D3)]\label{ditre} $(\RN,d)$ has the {\em segment property}, i.e., for every $x,y\in\RN$ there exists a continuous path $\gamma: [0,1]\to\RN$ such that $\gamma(0)=x$, $\gamma(1)=y$ and
$$d(x,y)=d(x,\gamma(t))+d(\gamma(t),y)\quad\forall t\in [0,1].$$
\end{enumerate}
\begin{remark}
Global Gaussian estimates as in \eqref{bounds} for the hypoelliptic operators of H\"ormander-type in \eqref{nondivmod} have been proved in \cite{BUt, BBLU}. More precisely, such estimates are obtained for an extended operator (outside $D\times]T_1,T_2[$) with respect to a Carnot-Carath\'eodory metric satisfying (\hyperref[diuno]{D1})--(\hyperref[ditre]{D3}). Properties $(i)-(ii)$ follow as well from the results in \cite{BBLU} (see also \cite{LP99} and Remark \ref{r12}).\\
This is the reason why we can apply our results to the class \eqref{nondivmod}, provided that we consider the relevant bounded open sets $\Omega$ to be compactly contained in $D\times]T_1,T_2[$ (see \eqref{iffHo}).
\end{remark}

Under the above assumptions the operator $\H$ endows the strip $S$ with a structure of $\beta$-harmonic space satisfying the Doob convergence property, see \cite[Theorem 3.9]{LU}. As a consequence, for any bounded open set $\O$ with $\overline{\O}\subseteq S$, the Dirichlet problem
$$\begin{cases}
\H u= 0 \text{ in }\Omega,  \\
u|_{\de\Omega}=\phi
\end{cases}$$
has a generalized solution $H_\phi^\Omega$, in the Perron-Wiener sense, for every continuous function $\phi:\de\Omega\rightarrow\R$. A point $z_0\in\de\O$ is called $\H$-regular if $\lim_{z\rightarrow z_0}{H_\phi^\Omega(z)}=\phi(z_0)$ for every $\phi\in C(\de\O)$. The main result of this paper is the following Wiener-Landis test for the $\H$-regularity of the boundary points of $\O$.

If $z_0\in\partial\Omega$ and $\lambda\in ]0,1[$ are fixed, we define for $k\in\N$
\begin{equation}\label{defOmk}
\O_{k}^c(z_0)=\left\{z \in S\smallsetminus\O \,:\,\left(\frac{1}{\l}\right)^{k\log{k}}\leq\Gamma(z_0,z) \leq\left(\frac{1}{\l}\right)^{(k+1)\log{(k+1)}}\right\} \cup \{z_0\}.
\end{equation}
\begin{theorem}\label{iff} Let $\Omega$ be a bounded open set with $\overline{\O}\subseteq S$, and let $z_0\in\partial\Omega.$ Then $z_0$ is $\H$-regular for $\de \Omega$  if and only if
\begin{equation} \label{iffeqn} \sum_{k=1}^{\infty}{V_{\Omega^c_k(z_0)}(z_0)}=+\infty.\end{equation} 
\end{theorem}
Here and in what follows, if $F$ is a compact subset of $\RNu$, $V_F$ will denote the {\it $\H$-balayage} of $F$ (see Section \ref{sec2} below for details).
\begin{remark}Thanks to Theorem \ref{reprev} below, we can write \eqref{iffeqn} as
$$\sum_{k=1}^{\infty}{\Gamma\ast\mu_{\Omega^c_k(z_0)}(z_0)}=+\infty.$$
\end{remark}

\begin{remark}\label{remarkall}
We would like to comment on the choice of the exponent $\alpha(k)=k\log{k}$ in the definition \eqref{defOmk} of $\Omega^c_k(z_0)$. The superlinear growth of $\alpha(k)$ is crucial for our proof. On the other hand, the exact analogue of the Evans-Gariepy criterion would have required the sequence of level sets with $\alpha(k)=k$. This is why Theorem \ref{iff} is a Wiener criterion `\`a la Landis', who proved in \cite{La} a similar result for the heat equation with a suitable choice of $\alpha(k)$ growing fast at infinity. Here, we don't use the strategy of Landis. We use instead, as we mentioned, the strategy in \cite{KLT} which takes ideas from \cite{L73, L77}. In \cite{KLT} it appears the same choice $\alpha(k)=k\log{k}$ as in Theorem \ref{iff}. We feel it is interesting to remark that we can get the same accuracy in the result in the present situation (not without an additional effort) where we know just two-sided Gaussian bounds on $\Gamma$ (and not an explicit expression). In this respect, we mention that in \cite{GL} the authors were able to prove the Evans-Gariepy-Wiener criterion in the case of smooth uniformly parabolic operators in divergence form for which the fundamental solution is not explicit: they were able to treat such a case by making crucial use of a refined Gaussian expansion of the fundamental solution in terms of the underlying geodesic Riemannian distance. A sub-Riemannian analogue of this noteworthy expansion is currently not available (to the best of our knowledge) for equations as in \eqref{mod}.
\end{remark}

\noindent{\bf Plan of the paper}. In Section \ref{sec2} we introduce suitable mean-value operators and we prove the everywhere identification of the balayages with the Riesz potentials. As a intermediate step we also prove a reproduction formula for the fundamental solution $\Gamma$. In Section \ref{sec3} we prove first the sufficient and then the necessary condition for the $\H$-regularity in Theorem \ref{iff}. To this aim, the crucial bound for the ratio $\frac{\Gamma(z,\zeta)}{\Gamma(z_0,\zeta)}$ is performed via H\"older-type estimates in Lemma \ref{rationew}, where $z, \zeta$ move in special subregions of $\O_{k}^c(z_0)$. The construction of such regions, denoted by $F^i_k$, and the proof of their needed properties are quite delicate (see \eqref{defFk}, see also Lemma \ref{findtstar} and \ref{sommadeglialtri}) and take a big part of Section \ref{sec3}. In Section \ref{sec4} we provide in Corollary \ref{infondo} a necessary and a sufficient condition for $\H$-regularity (different from each other) involving a suitable capacity of the compact sets $\Omega^c_k(z_0)$, and we then deduce a regularity criterion in terms of the Lebesgue measure of $\Omega^c_k(z_0)$ in Corollary \ref{infondobis}. Finally, we apply such regularity test to the model case of heat operators in Carnot groups by establishing in Corollary \ref{infondotris} a sharp geometric criterion for the regularity under an exterior $(\log\log)$-paraboloid condition.

\section{Balayages as potentials}\label{sec2}

The hypotheses mentioned in the Introduction allow in particular to exploit the results in \cite{LU}. For example, to our purposes, it is crucial the notion of balayage which yields various characterizations of the regularity of boundary points (see, e.g., \cite[Theorem 4.6]{LU}). We recall here the definition for the reader's convenience, together with other related notions of classical potential theory.

If $O\subseteq S$ is an open set, we say that a function $u:O\to ]-\infty,\infty]$ is $\H$-superharmonic in $O$ if $u$ is lower semi-continuous, it is finite in a dense subset of $O$, and
$$u\geq H^V_\varphi  \text{ in }V\,\,\,\,\forall \phi\in C(\partial V)\,\,\mbox{ with }\, \varphi\leq u|_{\partial V}$$
and for every $\H$-regular open set $V$ compactly contained in $O$. A bounded open set $V$ is called $\H$-regular if we can solve in a classical sense the Dirichlet problem related to $\H$ in $V$ for any continuous boundary datum (such $\H$-regular sets form a basis for the Euclidean topology). We use the notations $\overline{\HH}(O)$ for the set of $\H$-superharmonic functions in $O$. For a given a compact set $F\subseteq S$, we denote $W_F=\inf\{v\in\overline{\HH}(S)\,:\,v\ge 0 \text{ in }S,\ v\ge 1 \text{ in }F\}$ and we define the ($\H$-){\em{balayage}} potential of $F$ as
\begin{equation}\label{defbal}
V_F(z)=\liminf_{\zeta\to z}W_F(\zeta),\qquad z\in S.
\end{equation}
Here and in what follows we agree to let $\liminf_{\zeta\to z}w(\zeta)=\sup_{V\in\mathcal{U}_z}(\inf_{V} w)$ being $\mathcal{U}_z$ a basis of neighborhoods of $z$. We know from \cite[Proposition 8.3]{LU} that
\begin{equation}\label{quasidap}
V_F(z)=\Gamma\ast\mu_F (z)\quad\mbox{ for almost every point }z\in S \,\,\,(\mbox{and everywhere in }S\smallsetminus\de F),
\end{equation}
where $\mu_F$ denotes the Riesz-measure of $V_F$, i.e. the unique Radon measure in $S$ such that $\H V_F=-\mu_F$ in the sense of distributions. We recall that $\mu_F$ is a nonnegative measure with support in $F$. In this work we are going to prove that the equality \eqref{quasidap} holds at \emph{every} point of $S$. The validity of such representation in $\de F$ will be in fact crucial in the proof of Theorem \ref{iff}.
\begin{theorem}\label{reprev}
We have
$$V_F(z)=\Gamma\ast\mu_F (z)\quad \mbox{ for \emph{every} }z\in S\mbox{ (not only almost everywhere)}.$$
\end{theorem}

In the proof of this result we use \emph{mean-value representation formulas} for $C^2$-functions. To this aim, let us write the operator in the following form
$$\H=\div_x\left(Q(x,t)\nabla_x\right) + Y -\de_t.$$
For any $r>0$ and $z\in S$, we introduce the following mean-value operator
\begin{eqnarray*}
&& M_r u(z)= M^1_r u(z) + M^2_r u(z) = \int_{\P}{ E^1_r(z,\zeta) u(\zeta)\,{\rm{d}}\zeta} + \int_{\P}{ E^2_r(z,\zeta) u(\zeta)\,{\rm{d}}\zeta}= \\
&& \frac{1}{r}\int_{\P}{\Gamma^{-2}(z,\zeta)\left\langle Q(\zeta)\nabla_\xi\Gamma(z,\zeta),\nabla_\xi\Gamma(z,\zeta) \right\rangle u(\zeta)\,{\rm{d}}\zeta} + \frac{1}{r}\int_{\P}{\div(Y)(\zeta)\log{\left(r\Gamma(z,\zeta)\right)} u(\zeta)\,{\rm{d}}\zeta}
\end{eqnarray*}
where
$$\P=\left\{\zeta \in S\,:\,\Gamma(z,\zeta)>\frac{1}{r}\right\}.$$
We explicitly remark that $E^1_r(z,\zeta)\geq 0$ whereas $E^2_r(z,\zeta)$ may change sign: this is due to the presence of $Y$ in the structure of $\H$. For this reason, we also introduce
\begin{equation}\label{defNr}
N_ru(z)= \int_{\P}{ |E^1_r(z,\zeta) + E^2_r(z,\zeta)|u(\zeta)\,{\rm{d}}\zeta}.
\end{equation}
If $u$ is a $C^2$-function in a neighborhood of a fixed point $z\in S$ and $r$ is small enough, we have
\begin{equation}\label{mean}
u(z)= M_ru(z) - \frac{1}{r}\int_0^r\int_{\mathcal{P}(z,\rho)}{\left(\Gamma(z,\zeta)-\frac{1}{\rho}\right)\H u(\zeta)\,{\rm{d}}\zeta\,{\rm{d}}\rho}.
\end{equation}
The above formula can be proved by arguing essentially as in \cite[Theorem 1.5]{LP99} and using the Gaussian estimates \eqref{bounds}. We observe here that the hypothesis \eqref{Hbissa}, together with the Gaussian estimates, implies that $\int_{\RN}{\Gamma(z,(\xi, t-\e))\,{\rm{d}}\xi}\rightarrow 1$ as $\e\rightarrow 0^+$ (as shown in Remark \ref{eqass}): this is enough to complete the proof of the mean value formulas without knowing that $\int_{\RN}{\Gamma(z,\zeta)\,{\rm{d}}\xi}$ is identically $1$ for any $\tau< t$ (as used in \cite[page 311]{LP99}; see also \cite{KLTade}). 

\begin{remark}\label{eqass}
If \eqref{bounds} holds, then the assumption \eqref{Hbissa} is equivalent to
\begin{equation}\label{Gbissa}
\textstyle \int_{\RN}{\Gamma(x,t,\xi,\tau)\,{\rm d}\xi}\to 1
\end{equation}
as $x\to x_0$, $t\searrow\tau\in]T_1,T_2[$ and also as $x\to x_0$, $\tau\nearrow t\in]T_1,T_2[$.
\end{remark}
\begin{proof}
We first recall that there exists a constant $\beta\geq 1$ such that
\begin{equation}\label{boundgamma}
\textstyle \beta^{-1}\leq \int_{\RN}{\Gamma(x,t,\xi,\tau)\,{\rm d}\xi}\leq \beta \quad\mbox{ for every } x \mbox{ and for every }\tau<t.
\end{equation} 
This follows from the Gaussian estimates since we know from \cite[Proposition 2.4]{LU} that 
\begin{equation}\label{besponen}
\frac{1}{\beta(a)}\leq \int_{\RN}{G_a(x,t,\xi,\tau)\,{\rm d}\xi}\leq \beta(a) \quad\mbox{ for every } x \mbox{ and for every }\tau<t.
\end{equation}
To prove that \eqref{Hbissa} implies \eqref{Gbissa}, we write
$$\textstyle \int_{\RN}{\Gamma(x,t,\xi,\tau)\,{\rm d}\xi} = \int_{\RN}{\Gamma(x,t,\xi,\tau)\,\phi_k(\xi)\,{\rm d}\xi} + \int_{\RN\smallsetminus B(x_0,k)}{\Gamma(x,t,\xi,\tau)\,\left(1-\phi_k(\xi)\right)\,{\rm d}\xi}$$
where $0\leq \phi_k \leq 1$ is a suitable sequence of $C_0$-cut-off functions equal to $1$ in $B(x_0,k)$. The second integral at the r.h.s. can be made arbitrarily small by picking a large $k$ using \eqref{boundgamma}, whereas the first integral tends to $\phi_k(x_0)=1$ by \eqref{Hbissa} respectively as $x\to x_0$, $t\searrow\tau$ or as $x\to x_0$, $\tau\nearrow t$.\\
On the other hand, in order to prove that \eqref{Gbissa} implies \eqref{Hbissa}, for any $\phi\in C_0$ we can write
\begin{eqnarray*}
&&\textstyle \int_{\RN}{\Gamma(x,t,\xi,\tau)\,\phi(\xi)\,{\rm d}\xi} - \phi(x_0)=  \phi(x_0)\left( \int_{\RN}{\Gamma(x,t,\xi,\tau)\,{\rm d}\xi} - 1 \right) \\
&+&\textstyle \int_{\RN\smallsetminus B(x_0,\delta)}{\Gamma(x,t,\xi,\tau)\,(\phi(\xi)-\phi(x_0))\,{\rm d}\xi} + \int_{B(x_0,\delta)}{\Gamma(x,t,\xi,\tau)\,(\phi(\xi)-\phi(x_0))\,{\rm d}\xi}.
\end{eqnarray*}
The first integral at the r.h.s. tends to $0$ by \eqref{Gbissa}. From the continuity of $\phi$ and \eqref{boundgamma}, the last integral can be made arbitrarily small by picking a small $\delta>0$. Lastly, for such a fixed $\delta$, the second integral tends to $0$ by the Gaussian estimates (as $t-\tau\rightarrow 0$ and $x\rightarrow x_0$).
\end{proof}
In what follows, we also set
$$\d(z,\zeta)=(d(x,\xi)^4+(t-\tau)^2)^{\tfrac{1}{4}},\quad  z=(x,t),\,\zeta=(\xi,\tau)\in S.$$
The relative {\em parabolic balls} are
$$\B(z,r)=\{\zeta\in S\,:\, \d(z,\zeta)<r\},\quad z\in S,\ r>0.$$
For the proof of Theorem \ref{reprev}, we need the following reproduction formula for the fundamental kernel $\Gamma$.

\begin{proposition}\label{repro}
We have
$$\Gamma(z,\eta)=\int_{\RN}{\Gamma(z,\zeta)\Gamma(\zeta,\eta)\,{\rm d}\xi}$$
for every $z=(x,t), \zeta=(\xi,\tau), \eta=(y,s)\in S$ with $t>\tau>s$.
\end{proposition}
\begin{proof}
Fix $\eta$ and $\tau>s$. For any $z$, let us denote $v(z)=\int_{\RN}{\Gamma(z,\zeta)\Gamma(\zeta,\eta)\,{\rm d}\xi}$. Both $v$ and $\Gamma(\cdot,\eta)$ are solutions in $\RN\times(\tau,T_2)$. Then, by the maximum principle \cite[Proposition 3.1]{LU}, in order to prove the statement it is enough to prove the following two facts: both $v$ and $\Gamma(\cdot,\eta)$ tend to $0$ as $\hat{d}(z,0)\rightarrow +\infty$; $v(z)\rightarrow \Gamma((x_0,\tau);\eta)$ as $z\rightarrow (x_0,\tau)$ with $t>\tau$. It is immediate to see that $\Gamma(\cdot,\eta)$ tends to $0$ at infinity by the Gaussian estimates \eqref{bounds} and the properties (\hyperref[diuno]{D1})-(\hyperref[ditre]{D2}). On the other hand, by \eqref{bounds}, we have
\begin{equation}\label{primagaustutto}
v(z)\leq \Lambda^2\int_{\RN\smallsetminus B(x,\frac{1}{2}d(x,y))}{G_{a_0}(z,\zeta)G_{a_0}(\zeta,\eta)\,{\rm d}\xi} + \Lambda^2 \int_{B(x,\frac{1}{2}d(x,y))}{G_{a_0}(z,\zeta)G_{a_0}(\zeta,\eta)\,{\rm d}\xi}.
\end{equation}
The first term in the right-hand side of \eqref{primagaustutto} can be bounded above exploiting the fact that 
\begin{equation}\label{dila25}
G_{a_0}(\zeta,\eta)\leq c \hat{d}^{-Q}(\zeta, \eta)\quad \mbox{ with }c=c(\eta, a_0) 
\end{equation}
which follows from \cite[Proposition 2.2 and Proposition 2.5]{LU}: thus we get
\begin{eqnarray*}
\int_{\RN\smallsetminus B(x,\frac{1}{2}d(x,y))}{G_{a_0}(z,\zeta)G_{a_0}(\zeta,\eta)\,{\rm d}\xi}&\leq& \frac{c}{(\tau-s)^{\frac{Q}{2}}} \int_{\RN\smallsetminus B(x,\frac{1}{2}d(x,y))}{G_{a_0}(z,\zeta)\,{\rm d}\xi}\\
&\leq& \frac{c}{(\tau-s)^{\frac{Q}{2}}} e^{-\frac{a_0}{8}\frac{d^2(x,y)}{t-\tau}}\int_{\RN\smallsetminus B(x,\frac{1}{2}d(x,y))}{G_{\frac{a_0}{2}}(z,\zeta)\,{\rm d}\xi},
\end{eqnarray*}
where in the last inequality we used the fact that $G_{a_0}(z,\zeta)=e^{-\frac{a_0}{2}\frac{d^2(x,\xi)}{t-\tau}}G_{\frac{a_0}{2}}(z,\zeta)$ and the relation $d(x,\xi)\geq \frac{1}{2}d(x,y)$. The last term in the right-hand side of \eqref{primagaustutto} can be bounded above noting that $B(x,\frac{1}{2}d(x,y))\subseteq \RN\smallsetminus B(y,\frac{1}{2}d(x,y))$ by triangle inequality and using again \eqref{dila25}: this yields
\begin{eqnarray*}
\int_{B(x,\frac{1}{2}d(x,y))}{G_{a_0}(z,\zeta)G_{a_0}(\zeta,\eta)\,{\rm d}\xi}&\leq& \int_{\RN\smallsetminus B(y,\frac{1}{2}d(x,y))}{G_{a_0}(z,\zeta)G_{a_0}(\zeta,\eta)\,{\rm d}\xi}\\
\leq c \int_{\RN\smallsetminus B(y,\frac{1}{2}d(x,y))}{G_{a_0}(z,\zeta)\hat{d}^{-Q}(\zeta, \eta)\,{\rm d}\xi}
&\leq& c\left(\frac{d(x,y)}{2}\right)^{-Q}\int_{\RN\smallsetminus B(y,\frac{1}{2}d(x,y))}{G_{a_0}(z,\zeta)\,{\rm d}\xi}.
\end{eqnarray*}
Inserting the previous two estimates for the terms in the right-hand side in \eqref{primagaustutto}, and using \eqref{besponen}, we infer 
\begin{eqnarray*}
v(z)&\leq& \frac{\Lambda^2 c}{(\tau-s)^{\frac{Q}{2}}} e^{-\frac{a_0}{8}\frac{d^2(x,y)}{t-\tau}}\int_{\RN\smallsetminus B(x,\frac{1}{2}d(x,y))}{\hspace{-0.1cm}G_{\frac{a_0}{2}}(z,\zeta)\,{\rm d}\xi} + \frac{\Lambda^2 c 2^Q}{\left(d(x,y)\right)^Q}\int_{\RN\smallsetminus B(y,\frac{1}{2}d(x,y))}{\hspace{-0.1cm}G_{a_0}(z,\zeta)\,{\rm d}\xi} \\
&\leq& \beta(\frac{a_0}{2})\frac{ \Lambda^2 c}{(\tau-s)^{\frac{Q}{2}}} e^{-\frac{a_0}{8}\frac{d^2(x,y)}{T_2-T_1}} + \beta(a_0)\frac{\Lambda^2 c 2^Q}{\left(d(x,y)\right)^Q},
\end{eqnarray*}
which goes to $0$ as $z$ goes to $\infty$. We are left to prove that $v(z)\rightarrow \Gamma((x_0,\tau);\eta)$ as $z\rightarrow (x_0,\tau)$ with $t>\tau$. We can write
$$v(z)= \int_{\RN}{\Gamma(z,\zeta)\Gamma(\zeta,\eta)\phi_k(\xi)\,{\rm d}\xi} + \int_{\RN\smallsetminus B(x_0,k)}{\Gamma(z,\zeta)\Gamma(\zeta,\eta)(1-\phi_k(\xi))\,{\rm d}\xi}.$$
We can then argue similarly to Remark \eqref{eqass}: the second integral can be made arbitrarily small for large $k$ using \eqref{dila25} and \eqref{boundgamma}, whereas the first integral tends to $\Gamma((x_0,\tau);\eta)\phi_k(x_0)=\Gamma((x_0,\tau);\eta)$ as $x\to x_0$, $t\searrow\tau$ by \eqref{Hbissa}.
\end{proof}

We are finally ready to provide the proof of Theorem \ref{reprev}.
\begin{proof}[Proof of Theorem \ref{reprev}]
We first prove that $V_F(z)\leq \Gamma\ast\mu_F (z)$ for any fixed $z\in S$. Since $V_F$ and $\Gamma\ast\mu_F$ are nonnegative functions, we can assume $V_F(z)>0$ and $\Gamma\ast\mu_F (z)<+\infty$. By the lower semicontinuity of $V_F$ we know that, for any $0<\e<V_F(z)$, there exists $r_\e>0$ such that
$V_F(\zeta)\geq V_F(z)-\e\geq 0$ for all $\zeta\in \B(z,r_\e)$. The upper bound in \eqref{bounds} implies that $\Gamma(z,\cdot)$ is bounded from above in $S\smallsetminus\B(z,r_\e) $ by a positive constant $M_\e$. Recalling the definition of $\mathcal{P}(z,\cdot)$ and choosing $\bar{r}_\e=M_\e^{-1}$, we obtain $\mathcal{P}(z, \bar{r}_\e)\subseteq \B(z,r_\e)\Subset S$. Since we know that $M_r(1)\equiv 1$ from \eqref{mean} and $M_r(1)\leq N_r(1)$ by definition, for all $0<r<\bar{r}_\e$ we have
$$V_F(z)-\e = \left(V_F(z)-\e\right)M_r(1)(z) \leq \left(V_F(z)-\e\right)N_r(1)(z) = N_r(\left(V_F(z)-\e\right))(z).$$
On the other hand, since $N_r$ is monotone having a nonnegative kernel, while $V_F$ and $\Gamma\ast \mu_F$ have the same average being equal almost everywhere by \eqref{quasidap}, we then get
$$N_r(\left(V_F(z)-\e\right))(z)\leq   N_r(V_F)(z)=N_r(\Gamma\ast\mu_F)(z).$$
We now claim that there exists a nonnegative function $\delta(r)$ which vanishes as $r\rightarrow 0^+$ such that
\begin{equation}\label{cllaim}
N_r\left(\Gamma\ast \mu_F\right)(z)\leq \left(1+\delta(r)\right)\Gamma\ast \mu_F (z).
\end{equation}
Once this is established, collecting the above inequalities we obtain
$$V_F(z)-\e\leq \left(1+\delta(r)\right)\Gamma\ast \mu_F (z)\qquad\mbox{for all }0<r<\bar{r}_\e.$$
Letting $r\rightarrow 0^+$ and then $\e\rightarrow 0^+$, we deduce $V_F(z)\leq \Gamma\ast\mu_F (z)$ as desired. We are thus left with the proof of the claim. Denoting by $m=\max_{\overline{\P}}{|{\div(Y)}|}$, for sufficiently small $r>0$ we have
\begin{eqnarray*}
&&\int_{\P}{ |E^2_r(z,\zeta)|\left(\Gamma\ast\mu_F\right)(\zeta)\,{\rm{d}}\zeta}\leq \frac{m}{r}\int_{\P}{\log{\left(r\Gamma(z,\zeta)\right)}\left(\int_{S}{\Gamma(\zeta,\eta)\,{\rm{d}}\mu_F(\eta)}\right)\,{\rm{d}}\zeta}\\
&&=m\int_{S}{\left( \int_{\P}{\frac{\log{\left(r\Gamma(z,\zeta)\right)}}{r}\Gamma(\zeta,\eta)\,{\rm{d}}\zeta} \right)\,{\rm{d}}\mu_F(\eta)}\leq m\int_{S}{\left( \int_{\P}{\Gamma(z,\zeta)\Gamma(\zeta,\eta)\,{\rm{d}}\zeta} \right)\,{\rm{d}}\mu_F(\eta)}\\
&&\leq m\int_{S}{\int_{t_r}^t{\left( \int_{\RN}{\Gamma(z,(\xi,\tau))\Gamma((\xi,\tau),\eta)\,{\rm{d}}\xi} \right)\,{\rm{d}\tau}}\,{\rm{d}}\mu_F(\eta)}
\end{eqnarray*}
where $t_r:=\min\left\{t'\,:\,(x',t')\in \overline{\P}\right\}$. From the reproduction property of $\Gamma$ in Proposition \ref{repro}, we then infer 
\begin{equation}\label{mdueabs}
\int_{\P}{ |E^2_r(z,\zeta)|\left(\Gamma\ast\mu_F\right)(\zeta)\,{\rm{d}}\zeta}\leq m\max_{(x',t')\in\overline{\P}}{|t-t'|} \cdot \left(\Gamma\ast\mu_F\right)(z)<+\infty.
\end{equation}
Let us now approximate $\Gamma\ast\mu_F$ with an increasing sequence of nonnegative $C^2$-functions $u_k$ such that $\H u_k\leq 0$ and $u_k\rightarrow \Gamma\ast\mu_F$ pointwise. This can be done using for example the same argument in \cite[page 307]{LP99}. By the mean-value formula for $C^2$-functions \eqref{mean} we immediately get
$$u_k(z)\geq M_r\left(u_k\right)(z)=M^1_r\left(u_k\right)(z) + M^2_r\left(u_k\right)(z).$$
On the other hand, as $k\rightarrow+\infty$, $M^1_r\left(u_k\right)(z)\rightarrow M^1_r\left(\Gamma\ast\mu_F\right)(z)$ by Beppo-Levi's theorem and $M^2_r\left(u_k\right)(z)\rightarrow M^2_r\left(\Gamma\ast\mu_F\right)(z)$ by dominated convergence recalling that $|E^2_r(z,\cdot)u_k|\leq |E^2_r(z,\zeta)|\Gamma\ast\mu_F\in L^1(\P)$ by \eqref{mdueabs}. This yields
$$\left(\Gamma\ast\mu_F\right)(z)\geq M^1_r\left(\Gamma\ast\mu_F\right)(z) + M^2_r\left(\Gamma\ast\mu_F\right)(z)=M_r\left(\Gamma\ast\mu_F\right)(z).$$
In particular
$$M^1_r\left(\Gamma\ast\mu_F\right)(z) \leq \left(\Gamma\ast\mu_F\right)(z) - M^2_r\left(\Gamma\ast\mu_F\right)(z)\leq \left(\Gamma\ast\mu_F\right)(z) +\int_{\P}{ |E^2_r(z,\zeta)|\left(\Gamma\ast\mu_F\right)(\zeta)\,{\rm{d}}\zeta}.$$
Therefore, recalling the definition of $N_r$ in \eqref{defNr} and making use of \eqref{mdueabs}, we get
\begin{eqnarray*}
&&N_r\left(\Gamma\ast\mu_F\right)(z) \leq M^1_r\left(\Gamma\ast\mu_F\right)(z) +  \int_{\P}{ |E^2_r(z,\zeta)|\left(\Gamma\ast\mu_F\right)(\zeta)\,{\rm{d}}\zeta}\\
&&\leq \left(\Gamma\ast\mu_F\right)(z) +2\int_{\P}{ |E^2_r(z,\zeta)|\left(\Gamma\ast\mu_F\right)(\zeta)\,{\rm{d}}\zeta}\\
&&\leq \left(1+2\max_{\overline{\P}}{|{\div(Y)}|}\max_{(x',t')\in\overline{\P}}{|t-t'|}\right)\left(\Gamma\ast\mu_F\right)(z).
\end{eqnarray*}
This proves the claim \eqref{cllaim} recalling that $\P$ shrinks to $\{z\}$ as $r\rightarrow 0^+$ by the Gaussian estimates in \eqref{bounds}.\\
We now turn to the proof of the opposite inequality $V_F\geq \Gamma\ast\mu_F$. Consider any $v\in\overline{\HH}(S)$ with $v\geq 0$ in $S$ and $v\geq 1$ in $F$. Then $v-\Gamma\ast\mu_F\in\overline{\HH}(S\smallsetminus F)$ being $\Gamma\ast\mu_F$ $\H$-harmonic outside $F$. Since $\Gamma\ast\mu_F\leq 1$ in $S$ (see, e.g., \cite[Proposition 8.3]{LU}), we have 
$$\liminf_{S\smallsetminus F\ni \eta\rightarrow\zeta}{(v-\Gamma\ast\mu_F)(\eta)}\geq v(\zeta)-1\geq0\quad\forall\zeta\in\de F.$$ Moreover $\liminf_{d(x,0)\rightarrow+\infty}{(v-\Gamma\ast\mu_F)(x,t)}\geq 0$ by \eqref{bounds}. This implies $v\geq\Gamma\ast\mu_F$ in $S\smallsetminus F$ by the minimum principle in \cite[Proposition 3.10]{LU}. On the other hand, $v\geq 1 \geq\Gamma\ast\mu_F $ also inside $F$. Thus $v\geq\Gamma\ast\mu_F$ for all $v$ as above. As a consequence $W_F\geq\Gamma\ast\mu_F$ by definition of $W_F$, and hence
$$V_F(z)=\liminf_{\zeta\rightarrow z}W_F(\zeta)\geq\liminf_{\zeta\rightarrow z}\Gamma\ast\mu_F(\zeta)\geq\Gamma\ast\mu_F(z)\qquad\mbox{for all }z\in S$$
by the lower semicontinuity of $\Gamma\ast\mu_F$.
\end{proof}

\section{Proof of the main result}\label{sec3}

In this section we set for the sake of brevity the notation
\begin{equation}\label{defalfa}
\alpha(k)=k\log{k}. 
\end{equation}
We are going to make a repeated use of the following simple properties of the sequence $\alpha(k)$:
\begin{itemize}
\item[-] $k\mapsto \alpha(k)$ is monotone increasing and tends to $+\infty$;
\item[-] $k\mapsto \alpha(k+p)-\alpha(k)$ is monotone increasing and tends to $+\infty$, for any $p\in\N$.
\end{itemize}
We consider a bounded open set $\O$ with closure contained in $S$. For any fixed $z_0\in\de\O$, we recall that
$$\O_{k}^c(z_0)=\left\{z \in S\smallsetminus\O \,:\,\l^{-\alpha(k)}\leq\Gamma(z_0,z) \leq\l^{-\alpha(k+1)}\right\} \cup \{z_0\}$$
where $k\in\N$ and $\l\in(0,1)$. We also denote
$$E_k(z_0)=\left\{z \in S \,:\,\Gamma(z_0,z)\geq \l^{-\alpha(k)}\right\} \cup \{z_0\}.$$
We start noticing that, by the Gaussian estimates \eqref{bounds} and the property (\hyperref[diuno]{D1}), the sets $E_k(z_0)$ have non-empty interior for all $k\in\N$. Moreover, we remark that the sets $\O_{k}^c(z_0)\subseteq E_k(z_0)$ shrink to the point $z_0$ as $k$ grows. More precisely, by \eqref{bounds} and the doubling property (\hyperref[diuno]{D2}), we get
\begin{equation}\label{shrink}
\forall r>0 \,\,\, \exists \bar{k}=\bar{k}(\lambda,\Lambda,a_0,c_d, x_0)\, \mbox{ such that }\,\O_{k}^c(z_0)\subseteq \left(\B(z_0,r)\cap\{t\leq t_0\}\right)\smallsetminus\O\,\,\,\mbox{ for all }k\geq\bar{k}.
\end{equation}

We first prove the sufficient condition for the regularity in Theorem \ref{iff}. Let us assume that, for some fixed $\l\in (0,1)$, we have
$$\sum_{k=1}^{\infty}{V_{\Omega^c_k(z_0)}(z_0)}=+\infty.$$
Hence, for any $q\in\N$, there has to exist at least one $i\in\{0,\ldots,q-1\}$ such that
\begin{equation}\label{plusinf}
\sum_{k=1}^{\infty}{V_{\Ockq}(z_0)}=+\infty.
\end{equation}
We want to exploit \eqref{plusinf} for a suitable choice of a constant $q$ which we are now going to fix once for all. Let us denote by
\begin{equation}\label{defQh}
Q_\beta=2\left(\frac{Q}{\beta}+1\right),
\end{equation}
where $\beta\in (0,1)$ is the structural H\"older exponent appearing in the H\"older estimate for the solution to $\H v=0$ (we refer the reader to \eqref{Hest} below). We then fix $q\in\N$ such that 
\begin{equation}\label{defq}
q\geq q_0:=Q_\beta+\frac{m}{\log{\left(\frac{1}{\l}\right)}}\qquad\mbox{ for some constant }m. 
\end{equation}
To be precise, we can choose
$$m=\max{\left\{Q_\beta+1,\frac{\log{\left(4 c^2_d\Lambda^2\right)}}{\log\left(1+Q_\beta^{-1}\right)}, \frac{\log{\left(2 c_d e^{\frac{Q}{2}}\right)}}{\log\left(1+Q_\beta^{-1}\right)}, \frac{\log{\left(2 c_d e^{\frac{a_0}{2}}\right)}}{\log\left(1+Q_\beta^{-1}\right)}, \frac{\log{\left(c_d20^{\frac{Q}{2}}\right)}}{\log\left(\frac{2Q_\beta}{Q_\beta+2}\right)}, \frac{\log{\left(c_d\left(\frac{10Q}{ea_0}\right)^{\frac{Q}{2}}\right)}}{\log\left(\frac{2Q_\beta}{Q_\beta+2}\right)} \right\}}.$$
Let us now pick $i\in\{0,\ldots,q-1\}$ satisfying \eqref{plusinf}. Denote
\begin{equation}\label{deftkappa}
\Tq=\max_{\zeta\in E_{kq+i}(z_0)}{t_0-\tau}.
\end{equation}
By \eqref{bounds} and the definition of $\O_{k}^c(z_0)$, we have that
\begin{equation}\label{defTkappa}
\sup_{\zeta \in \Ockq}{|B(x_0,\sqrt{t_0-\tau})|}\leq |B(x_0,\sqrt{\Tq})|\leq \L \l^{\alpha\left(kq+i\right)}.
\end{equation}
We also denote by
$$p=1+\left[\frac{q}{Q_\beta}\right]=1+\mbox{the integer part of }\frac{q}{Q_\beta}.$$
So we get, since $q>Q_\beta$, that
\begin{equation}\label{defpi}
\frac{q}{Q_\beta}\leq p\leq 1 + \frac{q}{Q_\beta} <\frac{q}{\frac{Q}{\beta}+1}. 
\end{equation}
We need the following lemma.
\begin{lemma}\label{findtstar}
For any $k\in\N$ there exists $\Tst\in (0,\Tq)$ such that
\begin{equation}\label{deftst}
\L \l^{\alpha(kq+p+i)}\leq \left|B\left(x_0,\sqrt{\Tst}\right)\right|\leq 2c_d\L \l^{\alpha(kq+p+i)}.
\end{equation}
\end{lemma}
\begin{proof}
We first claim that, for every $x_0\in\RN$ and $\sigma>0$, there exists $\rho>0$ such that
\begin{equation}\label{claimetric}
\sigma\leq \left|B\left(x_0,\rho\right)\right|\leq 2c_d\sigma.
\end{equation}
This follows in fact from the properties (\hyperref[diuno]{D1})-(\hyperref[diuno]{D2}) of the metric space $\left(\RN,d\right)$ we are working in. Let us prove \eqref{claimetric} in full details. Fix $x_0\in\RN$ and $\sigma>0$, and consider
$$\rho(\sigma)=\sup\{r>0\,:\, |B(x_0,r)|\leq\sigma\}<+\infty.$$
Since $|B(x_0,\rho(\sigma))|=\left|\bigcup_{r<\rho(\sigma)}B(x_0,r)\right|=\lim_{r\rightarrow\rho(\sigma)^-}{|B(x_0,r)|}$, we have $|B(x_0,\rho(\sigma))|\leq\sigma$. From the definition of $\rho(\sigma)$ and the doubling condition we deduce
$$\sigma < \left|B\left(x_0, \rho(\sigma)+\frac{1}{n}\right)\right|\leq c_d |B(x_0,\rho(\sigma))| \left(\frac{\rho(\sigma)+\frac{1}{n}}{\rho(\sigma)}\right)^{Q}\leq c_d \sigma \left(1+\frac{1}{n\rho(\sigma)}\right)^{Q}$$
for all $n\in\N$. We can then pick $\bar{n}\in\N$ such that $\sigma\leq \left|B\left(x_0,\rho\right)\right|\leq 2c_d\sigma$ for $\rho=\rho(\sigma)+\frac{1}{\bar{n}}$.\\
By applying \eqref{claimetric} for $\sigma=\L \l^{\alpha(kq+p+i)}$, we derive the existence of a positive $\Tst(=\rho^2)$ satisfying \eqref{deftst}. We need to prove that $\Tst<\Tq$. By the monotonicity of $r\mapsto |B(x_0,r)|$, it is enough to show that
\begin{equation}\label{checkTst}
\left|B\left(x_0,\sqrt{\Tst}\right)\right|\leq 2c_d\L \l^{\alpha(kq+p+i)} < \left|B\left(x_0,\sqrt{\Tq}\right)\right|.
\end{equation}
To prove \eqref{checkTst} we can exploit again \eqref{claimetric} for $\sigma=\frac{1}{2c_d\Lambda}\l^{\alpha(kq+i)}$. There exists then $\rho_k>0$ such that $\frac{1}{2c_d\Lambda}\l^{\alpha(kq+i)}\leq |B(x_0,\rho_k)|\leq \frac{1}{\Lambda}\l^{\alpha(kq+i)}$. Thus, by \eqref{bounds} and the inequality $|B(x_0,\rho_k)|\leq \frac{1}{\Lambda}\l^{\alpha(kq+i)}$, the point $(x_0,t_0-\rho_k^2)$ belongs to $E_{kq+i}(z_0)$. Hence we get
$$\left|B\left(x_0,\sqrt{\Tq}\right)\right|=\sup_{\zeta \in E_{kq+i}(z_0)}{|B(x_0,\sqrt{t_0-\tau})|}\geq |B(x_0,\rho_k)|\geq \frac{1}{2c_d\Lambda}\l^{\alpha(kq+i)}.$$ 
The proof of \eqref{checkTst} is then complete, provided that we have
\begin{equation}\label{oneforall}
\frac{1}{2c_d\Lambda}\l^{\alpha(kq+i)}>2c_d\L \l^{\alpha(kq+p+i)},\,\,\, \mbox{ i.e. }\,\,\, \left(\frac{1}{\l}\right)^{\alpha(kq+p+i)-\alpha(kq+i)}>4c^2_d\L^2.
\end{equation}
The last inequality holds true for every $k$ because of our choices for $q$ and $p$ in \eqref{defq} and \eqref{defpi}: as a matter of fact, by the monotonicity properties of $\alpha(\cdot)$ defined in \eqref{defalfa}, we have 
$$\alpha(kq+p+i)-\alpha(kq+i)\geq \alpha(q+p)-\alpha(q)\geq q\log\left(1+\frac{p}{q}\right)\geq q\log\left(1+Q_\beta^{-1}\right)>\frac{\log(4c^2_d\Lambda^2)}{\log{\left(\frac{1}{\l}\right)}}.$$
\end{proof}
The previous lemma allows us to split the set $\Ockq$ in two pieces. For any $k\in\N$ let us write
\begin{equation}\label{defFk}
\Ockq=\left(\Ockq\cap \{t\geq t_0-\Tst\}\right)\cup\left(\Ockq\cap \{t\leq t_0-\Tst\}\right): =F_k^{0,i}\cup F_k^i
\end{equation}
where the level $\Tst\in (0,\Tq)$ is the one given by Lemma \ref{findtstar} (satisfying \eqref{deftst}).\\
By \eqref{defTkappa}, \eqref{deftst}, and since $kq+p+i<q(k+1)+i$, we have
$$\left|B\left(x_0,\sqrt{T_{hq+i}}\right)\right|\leq \L \l^{\alpha(hq+i)} < \L \l^{\alpha(kq+p+i)} \leq \left|B\left(x_0,\sqrt{\Tst}\right)\right|\qquad\forall h,k\in\N,\,\,h\geq k+1.$$
This implies that, by construction,
$$\min_{(\xi,\tau)\in F^i_{k}}{(t_0-\tau)}\geq \Tst>T_{hq+i}\geq \max_{(x,t)\in F^i_{h}}{(t_0-t)}\qquad\forall h,k\in\N,\,\,h>k,$$
which says
\begin{equation}\label{sotto}
F_k^i\,\, \mbox{ lies strictly below }\,\,F_{h}^i\,\,\,\qquad\forall h,k\in\N,\,\,h>k.
\end{equation}
\begin{lemma}\label{sommadeglialtri}
Suppose \eqref{plusinf} holds. Then the compact sets $F_k^i$ defined by \eqref{defFk} satisfy
\begin{equation}\label{enough}
\sum_{k=1}^{\infty}{V_{F_k^i}(z_0)}=+\infty.
\end{equation}
\end{lemma}
\begin{proof} 
By the subadditivity of the $\H$-balayage potential (recall the definition in \eqref{defbal}) we have
$$V_{\Ockq}\leq V_{F_k^{0,i}} + V_{F_k^i}.$$
Since we know the validity of \eqref{plusinf}, then the desired \eqref{enough} will be a consequence of the following
\begin{equation}\label{claimfin}
\sum_{k=1}^{+\infty}{V_{F_k^{0,i}}(z_0)}<+\infty.
\end{equation}
To prove \eqref{claimfin}, we need to understand how $F_k^{0,i}$ shrinks to $\{z_0\}$ as $k$ grows.\\
For any $z=(x,t)\in F_k^{0,i}\subset \Ockq$ with $z\neq z_0$, by \eqref{bounds} we have
\begin{equation}\label{parabound}
d^2(x_0,x)\leq \frac{t_0-t}{a_0}\log{\left(\frac{\Lambda \l^{\alpha(kq+i)}}{\left|B\left(x_0,\sqrt{t_0-t}\right)\right|} \right)}.
\end{equation}
We recall that $z\in F_k^{0,i}$ implies by definition that $0< t_0-t \leq \Tst$, and we know from \eqref{defTkappa} and Lemma \ref{findtstar} that
$$\left|B\left(x_0,\sqrt{t_0-t}\right)\right|\leq \left|B\left(x_0,\sqrt{\Tst}\right)\right|\leq \left|B\left(x_0,\sqrt{\Tq}\right)\right|\leq \L \l^{\alpha\left(kq+i\right)}.$$
On the other hand, by \eqref{deftst} and the choices for $q$ and $p$ in \eqref{defq} and \eqref{defpi}, we also get (by arguing as for \eqref{oneforall})
\begin{equation}\label{lessthan}
\left|B\left(x_0,\sqrt{\Tst}\right)\right| \leq 2 c_d \Lambda \l^{\alpha(kq+p+i)} < \Lambda\l^{\alpha(kq+i)} \min\left\{\frac{1}{c_d},e^{-\frac{Q}{2}}, e^{-\frac{a_0}{2}}\right\}.
\end{equation}
These inequalities, together with the doubling condition (\hyperref[diuno]{D2}) which says that $|B(x_0,\sqrt{s})|s_1^{\frac{Q}{2}}\leq c_d \left|B\left(x_0,\sqrt{s_1}\right)\right|s^{\frac{Q}{2}}$ for any $0\leq s_1\leq s$, allow to bound the term in \eqref{parabound}. In particular we claim that
\begin{equation}\label{quasimonotone}
(t_0-t)\log{\left(\frac{\Lambda \l^{\alpha(kq+i)}}{\left|B\left(x_0,\sqrt{t_0-t}\right)\right|} \right)}\leq 2\Tst \log{\left(\frac{\Lambda \l^{\alpha(kq+i)}}{\left|B\left(x_0,\sqrt{\Tst}\right)\right|}\right)}.
\end{equation} 
To prove \eqref{quasimonotone} we can write
\begin{eqnarray*}
&&2\Tst \log{\left(\frac{\Lambda \l^{\alpha(kq+i)}}{\left|B\left(x_0,\sqrt{\Tst}\right)\right|}\right)} - (t_0-t)\log{\left(\frac{\Lambda \l^{\alpha(kq+i)}}{\left|B\left(x_0,\sqrt{t_0-t}\right)\right|} \right)} \\
&&=\Tst\log{\left(\frac{\Lambda \l^{\alpha(kq+i)}}{\left|B\left(x_0,\sqrt{\Tst}\right)\right|}\right)} + \left(\Tst - (t_0-t)\right)\log{\left(\frac{\Lambda \l^{\alpha(kq+i)}}{\left|B\left(x_0,\sqrt{\Tst}\right)\right|}\right)} \\
&&- (t_0-t)\log{\left(\frac{\left|B\left(x_0,\sqrt{\Tst}\right)\right|}{\left|B\left(x_0,\sqrt{t_0-t}\right)\right|} \right)}.
\end{eqnarray*}
By the doubling condition and the concavity of the logarithmic function, we have
$$\log{\left(\frac{\left|B\left(x_0,\sqrt{\Tst}\right)\right|}{\left|B\left(x_0,\sqrt{t_0-t}\right)\right|} \right)}\leq \log(c_d) + \frac{Q}{2}\log{\left(\frac{\Tst}{t_0-t}\right)}\leq \log(c_d) + \frac{Q}{2}\left(\frac{\Tst-(t_0-t)}{t_0-t}\right).
$$
Putting together the last two relations we get the proof of \eqref{quasimonotone} since
\begin{eqnarray*}
&&2\Tst \log{\left(\frac{\Lambda \l^{\alpha(kq+i)}}{\left|B\left(x_0,\sqrt{\Tst}\right)\right|}\right)} - (t_0-t)\log{\left(\frac{\Lambda \l^{\alpha(kq+i)}}{\left|B\left(x_0,\sqrt{t_0-t}\right)\right|} \right)} \\
&&\geq  \Tst\left(\log{\left(\frac{\Lambda \l^{\alpha(kq+i)}}{\left|B\left(x_0,\sqrt{\Tst}\right)\right|}\right)} - \log{ c_d}\right) \\
&&+ \left(\Tst - (t_0-t)\right)\left(\log{\left(\frac{\Lambda \l^{\alpha(kq+i)}}{\left|B\left(x_0,\sqrt{\Tst}\right)\right|}\right)} - \frac{Q}{2}\right)\geq 0,
\end{eqnarray*}
where we used \eqref{lessthan} and $0\leq t_0-t\leq \Tst$. Therefore, from \eqref{parabound} and \eqref{quasimonotone}, we deduce that
$$d^2(x_0,x)\leq \frac{2\Tst}{a_0}\log{\left(\frac{\Lambda \l^{\alpha(kq+i)}}{\left|B\left(x_0,\sqrt{\Tst}\right)\right|} \right)}$$
for every $z=(x,t)\in F_k^{0,i}$. Moreover, again from \eqref{lessthan}, we also have
$$t_0-t\leq \Tst \leq \frac{2\Tst}{a_0}\log{\left(\frac{\Lambda \l^{\alpha(kq+i)}}{\left|B\left(x_0,\sqrt{\Tst}\right)\right|} \right)}.$$
The last two inequalities tells us that
$$z\in \B(z_0,r_k)\quad\mbox{ with }\quad r^2_k=\sqrt{8}\frac{\Tst}{a_0}\log{\left(\frac{\Lambda \l^{\alpha(kq+i)}}{\left|B\left(x_0,\sqrt{\Tst}\right)\right|} \right)}.$$
This holds true for any $z\in F_k^{0,i}$, i.e. we have just proved that
\begin{equation}\label{inclu}
F_k^{0,i}\subseteq \B(z_0,r_k).
\end{equation}
The representation formula proved in Theorem \ref{reprev} (note that we cannot use \eqref{quasidap} since the point $z_0\in\de F_k^{0,i}$) allows to deduce
$$V_{F_k^{0,i}}(z_0)=\int_{F_k^{0,i}}{\Gamma(z_0,\zeta)\,\rm{d}\mu_{F_k^{0,i}}(\zeta)}\leq\left(\frac{1}{\l}\right)^{\alpha(kq+i+1)}\mu_{F_k^{0,i}}\left(F_k^{0,i}\right).$$
Moreover, from the monotonicity with respect to the inclusion in \eqref{inclu} and the results in \cite[Corollary 2.4 and Proposition 2.1]{LTU}, we know that
$$\mu_{F_k^{0,i}}\left(F_k^{0,i}\right)\leq C |B(x_0,r_k)|$$
for some structural positive constant $C$. This says that
$$
\sum_{k=1}^{+\infty}{V_{F_k^{0,i}}(z_0)}  \leq C \sum_{k=1}^{+\infty}{\left(\frac{1}{\l}\right)^{\alpha(kq+i+1)}|B(x_0,r_k)|}.
$$
Exploiting the expression we found for $r_k$, the doubling property, and \eqref{deftst}, we get
\begin{eqnarray*}
&&\sum_{k=1}^{+\infty}{V_{F_k^{0,i}}(z_0)}  \leq C c_d \left(\frac{\sqrt{8}}{a_0}\right)^{\frac{Q}{2}} \sum_{k=1}^{+\infty}{\left(\frac{1}{\l}\right)^{\alpha(kq+i+1)}\left|B\left(x_0,\sqrt{\Tst}\right)\right|\log^{\frac{Q}{2}}{\left(\frac{\Lambda \l^{\alpha(kq+i)}}{\left|B\left(x_0,\sqrt{\Tst}\right)\right|} \right)}}\\
&& \leq 2C c^2_d\Lambda \left(\frac{\log\left(\frac{1}{\lambda}\right)\sqrt{8}}{a_0}\right)^{\frac{Q}{2}} \sum_{k=1}^{+\infty}{\left(\frac{1}{\l}\right)^{\alpha(kq+i+1)-\alpha(kq+p+i)}\left(\alpha(kq+p+i)-\alpha(kq+i)\right)^{\frac{Q}{2}}}.
\end{eqnarray*}
Hence, \eqref{claimfin} is proved if we ensure the convergence of the series at the right-hand side. We thus notice that the sequences $\alpha(kq+p+i)-\alpha(kq+i+1)$ and $\alpha(kq+p+i)-\alpha(kq+i)$ (recalling \eqref{defalfa}) are asymptotically equivalent respectively to $(p-1)\log(kq+p+i)$ and $p\log(kq+p+i)$. Hence, the series under investigation behaves like
$$\sum_{k=1}^{+\infty}{\frac{1}{\left(kq+p+i\right)^{(p-1)\log\frac{1}{\l}}}\log^{\frac{Q}{2}}(kq+p+i)},$$
which is convergent since $p\geq\frac{q}{Q_\beta}>1+\frac{1}{\log(\frac{1}{\l})}$ {by \eqref{defpi} and \eqref{defq}}. This proves \eqref{claimfin}, and therefore the lemma.
\end{proof}

In the following lemma we finally determine the required bound for the ratio $\frac{\Gamma(z,\zeta)}{\Gamma(z_0,\zeta)}$ for $z\in F^i_h$ and $\zeta\in F^i_k$. We do this by exploiting the H\"older continuity of the solutions to $\H u=0$ proved in \cite{LU}. It is not surprising to infer estimates for the fundamental solution or for the relevant Green kernel by using H\"older-type estimates (see the related results in \cite[Proposition 7.4]{LU} and \cite[Lemma 3.3]{LTU}, see also \cite{Ug, TUjd}). The novelty in the present situation is due to the special regions $F^i_k$, and it is strictly related with the careful choices for $q$ and $p$ in \eqref{defq} and \eqref{defpi}. We have the following
\begin{lemma}\label{rationew}
There exists a positive constant $M_0$ such that
$$\Gamma(z,\zeta)\leq M_0 \Gamma(z_0,\zeta) \quad\forall\,z\in F^i_h,\,\forall\,\zeta\in F^i_k,\quad\forall\,h,k\in\N,\,\,h\neq k.$$
\end{lemma}
\begin{proof}
Fix any $h,k\in \N$ with $h\neq k$. If $h\leq k-1$, then by \eqref{sotto} and \eqref{bounds} we have $\Gamma(z,\zeta)=0$, and the statement is trivial. Thus, suppose $h\geq k+1$.\\
Let us notice that, for any $\zeta\in F^i_k$, the function $z\mapsto v_\zeta(z)= \frac{\Gamma(z,\zeta)}{\Gamma(z_0,\zeta)}$ is a solution to $\H v_\zeta=0$ outside $F^i_k$. We know from \cite[Theorem 7.2]{LU} that, if $u$ is a solution to $\H u=0$ in $\mathcal{C}_r=B(x_0,r)\times(t_0-r^2,t_0)$, then we have
\begin{equation}\label{Hest}
|u(z)-u(z')|\leq C_0 \max_{\overline{\mathcal{C}}_r}{|u|}\left(\frac{\d(z,z')}{r}\right)^\beta\qquad\forall\, z,z'\in \overline{\mathcal{C}}_{\frac{r}{2}}
\end{equation}
for some constant $C_0>0$ and $\beta\in (0,1)$. The constants $\beta$ and $C_0$ depend just on the constants $\Lambda, a_0, b_0$ in the Gaussian bounds \eqref{bounds} and on the doubling constant $c_d$ of the metric $d$. We want to use the estimate \eqref{Hest} for the function $v_\zeta$ defined above in the cylinder $\mathcal{C}_{r_k}$ with the choice
\begin{equation}\label{choicerk}
r^2_k=\frac{1}{5}\Tst.
\end{equation}
Since $\mathcal{C}_{r_k}\subset \RN \times (t_0-\Tst,t_0)$, we have in fact that $v_\zeta$ is a solution to $\H v_\zeta=0$ in $\mathcal{C}_{r_k}$. Let us then estimate $\max_{\overline{\mathcal{C}}_{r_k}}{|v_\zeta|}$. To do this, we use the definitions of the sets $F^i_k\subset \Ockq$ together with \eqref{bounds} which yield
\begin{eqnarray*}
0\leq v_\zeta(z)= \frac{\Gamma(z,\zeta)}{\Gamma(z_0,\zeta)} &\leq& \Lambda \frac{\lambda^{\alpha(kq+i)}}{|B(x,\sqrt{t-\tau})|}\leq \Lambda \frac{\lambda^{\alpha(kq+i)}}{|B(x_0,\sqrt{t-\tau}-d(x,x_0))|}\\
&\leq& \Lambda \frac{\lambda^{\alpha(kq+i)}}{|B(x_0,\frac{1}{2}\sqrt{t-\tau})|}  \qquad \mbox{ for any }\zeta\in F^i_k,\, z\in \overline{\mathcal{C}}_{r_k},
\end{eqnarray*}
where the last inequality is justified by the fact that $t-\tau\geq \Tst-r_k^2=4r_k^2$ by \eqref{choicerk}. From the inequality $t-\tau\geq \Tst-r_k^2=\frac{4}{5}\Tst$, the doubling condition and \eqref{deftst}, we also get
\begin{equation}\label{inizboundv}
v_\zeta(z)\leq \Lambda \frac{\lambda^{\alpha(kq+i)}}{\left|B(x_0,\frac{1}{\sqrt{5}}\sqrt{\Tst})\right|}\leq c_d 5^{\frac{Q}{2}} \frac{\lambda^{\alpha(kq+i)}}{\lambda^{\alpha(kq+p+i)}} \qquad \mbox{ for any }\zeta\in F^i_k,\, z\in \overline{\mathcal{C}}_{r_k}.
\end{equation}
We now claim that
\begin{equation}\label{Fhcista}
F^i_h\subseteq \overline{\mathcal{C}}_{\frac{r_k}{2}}\qquad\forall h\geq k+1. 
\end{equation}
To prove this claim, we first consider the inclusion
\begin{equation}\label{timecista}
[t_0-T_{hq+1}, t_0-T^*_{hq+i}]\subseteq \left[t_0-\frac{r_k^2}{4}, t_0\right]
\end{equation}
which is valid since $T_{hq+i}\leq \frac{r_k^2}{4}=\frac{1}{20}\Tst$. In fact, the doubling condition, \eqref{defTkappa} and \eqref{deftst} yield
\begin{equation}\label{doubletime}
\left(\frac{T_{hq+i}}{\Tst}\right)^{\frac{Q}{2}}\leq c_d\frac{\left|B(x_0,\sqrt{T_{hq+i}})\right|}{\left|B(x_0,\sqrt{\Tst})\right|}\leq c_d \frac{\lambda^{\alpha(hq+i)}}{\lambda^{\alpha(kq+p+i)}}\leq \left(\frac{1}{20}\right)^{\frac{Q}{2}}
\end{equation}
where the last inequality holds true because of our choices for $q$ and $p$ in \eqref{defq} and \eqref{defpi} since
\begin{eqnarray*}
\alpha(hq+i)-\alpha(kq+p+i)&\geq& \alpha(kq+q+i)-\alpha(kq+p+i)\geq \alpha(2q)-\alpha(q+p)\\
&\geq& q\log\left(\frac{2}{1+\frac{p}{q}}\right)\geq q\log\left(\frac{2Q_\beta}{Q_\beta+2}\right)>\frac{\log(c_d20^{\frac{Q}{2}})}{\log{\left(\frac{1}{\l}\right)}}.
\end{eqnarray*}
On the other hand, for any fixed $z\in F^i_h$, we have by \eqref{bounds} that
$$d^2(x_0,x)\leq \frac{t_0-t}{a_0}\log{\left(\frac{\Lambda \l^{\alpha(hq+i)}}{\left|B\left(x_0,\sqrt{t_0-t}\right)\right|} \right)}.$$
Using
$$
\left(\frac{t_0-t}{\Tst}\right)^{\frac{Q}{2}}\leq c_d\frac{\left|B(x_0,\sqrt{t_0-t})\right|}{\left|B(x_0,\sqrt{\Tst})\right|}
$$
together with \eqref{deftst} and the fact that $\max_{s\in[0,C]}{s^{\frac{2}{Q}}\log\left(\frac{C}{s}\right)}=\frac{Q}{2e}C^{\frac{2}{Q}}$, we deduce
\begin{eqnarray}\label{bellaquesta}
d^2(x_0,x)&\leq& \frac{\Tst c_d^{\frac{2}{Q}}}{a_0 \left|B(x_0,\sqrt{\Tst})\right|^\frac{2}{Q}}\left|B(x_0,\sqrt{t_0-t})\right|^{\frac{2}{Q}}\log{\left(\frac{\Lambda \l^{\alpha(hq+i)}}{\left|B\left(x_0,\sqrt{t_0-t}\right)\right|} \right)}\nonumber\\
&\leq& \frac{Q c_d^{\frac{2}{Q}}}{2ea_0}\frac{\lambda^{\frac{2}{Q}\alpha(hq+i)}}{\lambda^{\frac{2}{Q}\alpha(kq+p+i)}}\Tst.
\end{eqnarray}
Hence, we can affirm that
\begin{equation}\label{spacecista}
x\in B\left(x_0,\frac{1}{2}r_k\right)
\end{equation}
since with our choices for $q$ and $p$ in \eqref{defq} and \eqref{defpi} we have
$$\frac{Q c_d^{\frac{2}{Q}}}{2ea_0}\frac{\lambda^{\frac{2}{Q}\alpha(hq+i)}}{\lambda^{\frac{2}{Q}\alpha(kq+p+i)}}\Tst\leq \frac{r_k^2}{4}=\frac{1}{20}\Tst$$
because of the validity of the chain of inequalities
$$\alpha(hq+i)-\alpha(kq+p+i)\geq q\log\left(\frac{2Q_\beta}{Q_\beta+2}\right)>\frac{\log(c_d\left(\frac{10Q}{ea_0}\right)^{\frac{Q}{2}})}{\log{\left(\frac{1}{\l}\right)}}.$$
The claim \eqref{Fhcista} is thus a consequence of \eqref{timecista} and \eqref{spacecista}.\\
Therefore, for any $\zeta\in F^i_k$ we can apply in the cylinder $\mathcal{C}_{r_k}$ the estimate \eqref{Hest} to the function $v_\zeta$ with $z'=z_0$ and $z\in F^i_h$, and we get by \eqref{inizboundv} and \eqref{Fhcista}
\begin{equation}\label{daholderawiener}
|v_\zeta(z)-v_\zeta(z_0)|\leq C_0c_d 5^{\frac{Q}{2}} \frac{\lambda^{\alpha(kq+i)}}{\lambda^{\alpha(kq+p+i)}} \left(\frac{d^4(x,x_0)+(t_0-t)^2}{r^4_k}\right)^\frac{\beta}{4}.
\end{equation}
Keeping in mind that $z\in F^i_h$ and \eqref{choicerk}, we have by \eqref{bellaquesta}
$$
\frac{d^2(x,x_0)}{r_k^2}\leq \frac{5 c_d^{\frac{2}{Q}} Q}{2ea_0}\frac{\lambda^{\frac{2}{Q}\alpha(hq+i)}}{\lambda^{\frac{2}{Q}\alpha(kq+p+i)}},
$$
and by \eqref{doubletime}
$$
\left(\frac{t_0-t}{r_k^2}\right)^{\frac{Q}{2}}\leq\left(\frac{5T_{hq+i}}{\Tst}\right)^{\frac{Q}{2}}\leq 5^{\frac{Q}{2}} c_d \frac{\lambda^{\alpha(hq+i)}}{\lambda^{\alpha(kq+p+i)}}.
$$
Hence, recalling also that $v_\zeta(z_0)=1$, from \eqref{daholderawiener} we deduce that the following holds
\begin{eqnarray}\label{finallybounded}
&&|v_\zeta(z)-1|\leq C_0c_d^{1+\frac{\beta}{Q}} 5^{\frac{Q}{2}+\frac{\beta}{2}} \frac{\lambda^{\alpha(kq+i)}}{\lambda^{\alpha(kq+p+i)}} \left(\frac{Q^2}{4e^2a^2_0}\frac{\lambda^{\frac{4}{Q}\alpha(hq+i)}}{\lambda^{\frac{4}{Q}\alpha(kq+p+i)}} + \frac{\lambda^{\frac{4}{Q}\alpha(hq+i)}}{\lambda^{\frac{4}{Q}\alpha(kq+p+i)}}\right)^\frac{\beta}{4}\nonumber\\
&\leq& C \frac{\lambda^{\frac{\beta}{Q}\alpha(hq+i)}}{\lambda^{\frac{\beta}{Q}\alpha(kq+p+i)}} \frac{\lambda^{\alpha(kq+i)}}{\lambda^{\alpha(kq+p+i)}}\leq  C \frac{\lambda^{\frac{\beta}{Q}\alpha(kq+q+i)}}{\lambda^{\frac{\beta}{Q}\alpha(kq+p+i)}} \frac{\lambda^{\alpha(kq+i)}}{\lambda^{\alpha(kq+p+i)}}
\end{eqnarray} 
for all $z\in F^i_h$ and $\zeta\in F^i_k$, and for all $h\geq k+1$ (for some structural positive constant $C$). Our aim is to bound the right-hand side uniformly in $k$. In this respect, 
since $\alpha(n + s)-\alpha(n)$ is asymptotically equivalent {(recalling \eqref{defalfa})} to $s\log(n+s)$ as $n$ goes to $\infty$, we notice that 
$$\frac{\lambda^{\frac{\beta}{Q}\alpha(kq+q+i)}}{\lambda^{\frac{\beta}{Q}\alpha(kq+p+i)}} \frac{\lambda^{\alpha(kq+i)}}{\lambda^{\alpha(kq+p+i)}} \quad \mbox{ behaves like } \quad \frac{(kq+p+i)^{p\log(\frac{1}{\l})}}{(kq+q+i)^{\frac{\beta}{Q}(q-p)\log(\frac{1}{\l})}}$$
which is convergent to $0$ as $k\rightarrow+\infty$ since we have taken
$$q>\left(\frac{Q}{\beta}+1\right)p$$
in \eqref{defpi}. In particular, the terms in \eqref{finallybounded} are uniformly bounded by an absolute constant $M$. Therefore, by recalling the definition of $v_\zeta$ and \eqref{finallybounded}, we finally get
$$\frac{\Gamma(z,\zeta)}{\Gamma(z_0,\zeta)}=1 + v_\zeta(z)-1\leq 1 + M $$
for all $z\in F^i_h$ and $\zeta\in F^i_k$, and for all $h\neq k$.
\end{proof}
We are now ready to conclude the proof of the sufficient condition for the regularity in Theorem \ref{iff}. Assuming \eqref{plusinf}, we have defined in \eqref{defFk} a sequence of compact sets $\{F^i_k\}_{k\in\N}$ which are mutually disjoint by \eqref{sotto} and such that they shrink to the point $\{z_0\}$ as $k$ grows by \eqref{shrink}. Moreover, by Lemma \ref{sommadeglialtri} and Lemma \ref{rationew}, we have also that
$$\sum_{k=1}^{\infty}{V_{F_k^i}(z_0)}=+\infty,\qquad\mbox{and}$$
$$\sup{\left\{\frac{\Gamma(z,\zeta)}{\Gamma(z_0,\zeta)}\,:\,z\in F^i_h,\,\zeta\in F^i_k\right\}}\leq M_0\quad\forall h\neq k.$$
Therefore, we can proceed verbatim as in \cite[Lemma 6.1]{KLT} and we deduce that
\begin{equation}\label{farfrom0}
V_{\O'_r(z_0)}(z_0)\geq \frac{1}{2M_0}\quad\, \mbox{for every positive }r,
\end{equation}
where
\begin{equation}\label{defOprimoz}
\O'_r(z_0)=\left\{ z\in  S\smallsetminus \O \,:\, t\leq t_0, \quad \d(z,z_0)\leq r\right\}.
\end{equation}
We remark that in the proof it is needed the expression of the balayage in terms of its Riesz-representative as showed in Section \ref{sec2}. Once we have \eqref{farfrom0}, the $\H$-regularity of $z_0$ follows then from the characterization in \cite[Theorem 4.6]{LU}.

Let us turn to the proof of the necessary condition for the regularity in Theorem \ref{iff}. We assume then by contradiction that 
$$\sum_{k=1}^{\infty}{V_{\Omega^c_k(z_0)}(z_0)}<+\infty.$$
We want to prove that $z_0$ is not regular. For every $0<\e<\frac{1}{2}$ we have the existence of $L\in \N$ such that $\sum_{k=L}^{\infty}{V_{\Omega^c_k(z_0)}(z_0)}\leq\e$. For any $r>0$, recalling the definition of $\O'_r(z_0)$ in \eqref{defOprimoz}, we can write $\O'_r(z_0)=\O^L_r\cup \O_r^{*L}$ where
$$\O^L_r=\O'_r(z_0)\cap \left\{\Gamma(z_0,\cdot) \geq \left( \frac{1}{\lambda} \right)^{L \log L } \right\}\cup\{z_0\}\quad\mbox{and}\quad
\O_r^{*L}=\O'_r(z_0)\cap \left\{ \Gamma(z_0,\cdot) \leq \left( \frac{1}{\lambda} \right)^{L \log L } \right\}.$$
By definition $\O^L_r\subseteq \bigcup_{k=L}^\infty \Omega_k^c(z_0)$. Then, we get by the sub-additivity of the $\H$-balayage 
$$V_{\O'_r(z_0)} (z_0)\leq V_{\O_r^{*L}} (z_0) + V_{\O^L_r} (z_0) \leq V_{\O_r^{*L}} (z_0) + \sum_{k=L}^{\infty}{V_{\Omega^c_k(z_0)}(z_0)} \leq V_{\O_r^{*L}} (z_0) + \e.$$
This holds true for all $r>0$. By using the representation in Theorem \ref{reprev}, we thus have 
\begin{equation}\label{estbalepsplus}
V_{\O'_r(z_0)} (z_0)\leq \e+\int_{\O_r^{*L}}{\Gamma(z_0,\zeta)\,\rm{d}\mu_{\O_r^{*L}}(\zeta)}\leq \e+\l^{-L \log L }\mu_{\O_r^{*L}}\left(\O_r^{*L}\right).
\end{equation}
We stress that the representation result of Theorem \ref{reprev} is used here precisely at the point $z_0$ which belongs to $\de\O_r^{*L}$ for every $L$: the almost everywhere representation in \eqref{quasidap} would not be enough to deduce the previous estimate.\\
Since $\O_r^{*L}\subseteq \O'_r(z_0)\subset \overline{\B(z_0,r)}$, we can use \cite[Corollary 2.4 and Proposition 2.1]{LTU} to deduce that
$$\l^{-L \log L }\mu_{\O_r^{*L}}\left(\O_r^{*L}\right)\leq c \l^{-L \log L } |B(x_0,r)|<\frac{1}{2}.$$
where the last inequality follows from \cite[equation (2.2)]{LU} provided that $r$ is sufficiently small. Recalling \eqref{estbalepsplus}, this yields
$$V_{\O'_r(z_0)} (z_0)< \e + \frac{1}{2}<1 \quad\mbox{for small }r$$
which says that $z_0$ is not regular by \cite[Theorem 4.6]{LU}. The proof of Theorem \ref{iff} is thus complete.

\section{Corollaries and applications}\label{sec4}

As a first corollary of Theorem \ref{iff}, we want to read the sufficient and the necessary condition for the $\H$-regularity in terms of a series of capacitary terms. In contrast with the classical Wiener criteria, the necessary and sufficient conditions are here different. This is due to the presence of $\alpha(k)=k\log{k}$ in the definition of $\O_{k}^c(z_0)$ (see also Remark \ref{remarkall}).\\
For any compact set $F\subset S$, let us define the capacity of $F$ as
$$\mathrm{cap}_\H(F)=\mu_F(F),$$
where $\mu_F$ is the Riesz-measure associated to $V_F$.

\begin{corollary}\label{infondo}
Let $\Omega$ be a bounded open set with $\overline{\O}\subseteq S$, and $z_0 \in \de\Omega$. The following statements hold:
\begin{itemize} \item[$(i)$] if
$$\sum_{k=1}^{\infty}{\frac{\mathrm{cap}_\H (\Omega^c_k(z_0))}{\l^{k\log{k}}}}=+\infty$$
then $z_0$ is $\H$-regular;
\item[$(ii)$]  if $z_0$ is $\H$-regular then
$$\sum_{k=1}^{\infty}{\frac{\mathrm{cap}_\H (\Omega^c_k(z_0))}{\l^{(k+1)\log{(k+1)}}}}=+\infty.$$
\end{itemize}
\end{corollary}
\begin{proof}
By Theorem \ref{reprev} we can write, for any $k\in\N$,
$$V_{\Omega^c_k(z_0)}(z_0)=\Gamma\ast\mu_{\Omega^c_k(z_0)}(z_0)=\int_{\Omega^c_k(z_0)}\Gamma(z_0,\zeta)\,{\rm d}\mu_{\Omega^c_k(z_0)}(\zeta).$$
On the other hand, by definition of $\Omega^c_k(z_0)$ and of capacity, we trivially have
$$\frac{\mathrm{cap}_\H (\Omega^c_k(z_0))}{\l^{k\log{k}}}\leq \int_{\Omega^c_k(z_0)}\Gamma(z_0,\zeta)\,{\rm d}\mu_{\Omega^c_k(z_0)}(\zeta)\leq\frac{\mathrm{cap}_\H (\Omega^c_k(z_0))}{\l^{(k+1)\log{(k+1)}}} .$$
The proof of the statements is then straightforward from the characterization of Theorem \ref{iff}.
\end{proof}

Let us mention that other definitions of capacities related to $\H$ are possible and they are discussed, e.g., in \cite[Section 2]{LTU}. For example, one can deal with capacities with respect to Gaussian kernels $\G_{a}(\cdot,\cdot)$ which, in turn, can be estimated in terms of the Lebesgue measure (see \cite[Proposition 2.5]{LTU}). We can then obtain the following sufficient condition for the regularity which is more geometric and easier to be tested with respect to condition $(i)$ in Corollary \ref{infondo} (see also \cite[Corollary 1.3]{KLT}).

\begin{corollary}\label{infondobis}
Let $\Omega$ be a bounded open set with $\overline{\O}\subseteq S$, and $z_0 \in \de\Omega$. If, for some $\lambda\in(0,1)$, we have
$$\sum_{k=1}^{\infty}{\frac{|\Omega^c_k(z_0)|}{T_k \l^{k\log{k}}}}=+\infty$$
then the point $z_0$ is $\H$-regular for $\de\Omega$. In particular, $z_0$ is $\H$-regular for $\de\Omega$ if
\begin{equation}\label{suffcondmeasure}
\sum_{k=1}^{\infty}{\frac{|\Omega^c_k(z_0)|}{\l^{\frac{Q+2}{Q}k\log{k}}}}=+\infty.
\end{equation}
\end{corollary}
\begin{proof}
Recalling the notations we fixed in \eqref{deftkappa}, we know that
$$\Omega^c_k(z_0)\subset \RN\times [t_0-T_k, t_0].$$
Denoting by $\mathrm{cap}_{a_0}$ the capacity with respect the kernel $\G_{a_0}$, by \cite[Corollary 2.4]{LTU} and the monotonicity of  $\mathrm{cap}_{a_0}$ we get
\begin{eqnarray*}
\mathrm{cap}_\H (\Omega^c_k(z_0))&\geq& \frac{1}{c_0}\mathrm{cap}_{a_0}(\Omega^c_k(z_0))=\frac{1}{c_0T_k}\int_{t_0-T_k}^{t_0}\mathrm{cap}_{a_0}(\Omega^c_k(z_0))\,{\rm d}t\\
&\geq& \frac{1}{c_0T_k}\int_{t_0-T_k}^{t_0}\mathrm{cap}_{a_0}\left(\Omega^c_k(z_0)\cap\{\tau=t\}\right)\,{\rm d}t
\end{eqnarray*}
for some positive constant $c_0$. Moreover, we know from \cite[Proposition 2.5]{LTU} that there exists a positive constant $c$ such that $\mathrm{cap}_{a_0}\left(\Omega^c_k(z_0)\cap\{\tau=t\}\right)\geq c |\Omega^c_k(z_0)\cap\{\tau=t\}|$. Hence we have
$$\mathrm{cap}_\H (\Omega^c_k(z_0))\geq \frac{c}{c_0}\frac{1}{T_k}\int_{t_0-T_k}^{t_0}\left|\Omega^c_k(z_0)\cap\{\tau=t\}\right|\,{\rm d}t=\frac{c}{c_0}\frac{\left|\Omega^c_k(z_0)\right|}{T_k},$$
which says that
$$\sum_{k=1}^{\infty}{\frac{\mathrm{cap}_\H (\Omega^c_k(z_0))}{\l^{k\log{k}}}}\geq \frac{c}{c_0} \sum_{k=1}^{\infty}{\frac{|\Omega^c_k(z_0)|}{T_k \l^{k\log{k}}}}.$$
The first statement then follows from the sufficient condition in Corollary \ref{infondo}.\\
On the other hand, having in mind \eqref{defTkappa} and the doubling condition, we have
$$T_k\leq \left(c_d\frac{|B(x_0,\sqrt{T_k})|}{|B(x_0,1)|}\right)^{\frac{2}{Q}}\leq \left(\frac{c_d\L}{|B(x_0,1)|}\right)^{\frac{2}{Q}} \l^{\frac{2}{Q}k\log{k}}$$
at least for $k$ big enough (so that $T_k\leq 1$), which yields
$$\sum_{k}{\frac{|\Omega^c_k(z_0)|}{T_k \l^{k\log{k}}}}\geq \left(\frac{|B(x_0,1)|}{c_d\L}\right)^{\frac{2}{Q}} \sum_{k}{\frac{|\Omega^c_k(z_0)|}{\l^{\frac{Q+2}{Q}k\log{k}}}}$$
and complete the proof of the last part of the statement.
\end{proof}

The regularity criterion in the previous corollary is given in terms of subregions of the complementary set of $\O$ measured at different scales (keep in mind the definition of $\Omega^c_k(z_0)$ in \eqref{defOmk}). This is a recurring feature in potential theory. For example, for the class of operators $\H$ we are considering, it was proved in \cite[Theorem 4.11]{LU} that an exterior parabolic-cone density condition ensures the $\H$-regularity of the boundary point (see also \cite[Theorem 1.4]{LTU} for a $C^\alpha$-regularity result under the same condition). We want to show here that the criterion we have established in Corollary \ref{infondobis} is able, in some cases, to detect the $\H$-regularity of a boundary point in very sharp/subtle situations.\\
To see this, we specialize to the model case of heat operators in Carnot groups. Let us then assume that $\RN$ is endowed with a Carnot group structure $(\RN, \circ, D_\lambda)$, where $\circ$ denotes the group law operation and $D_\lambda$ the family of anisotropic dilations. We denote by $0$ the identity element of the group, and by $x^{-1}$ the inverse element of $x$. Let $X_1, \ldots, X_m$ be left-invariant vector fields which are $D_\lambda$-homogeneous of degree $1$ and form a basis for the first layer of the Lie algebra. We want to consider the following H\"ormander-type operators
\begin{equation}\label{homoper}
\H_0=\sum_{i=1}^{m}X^2_i-\de_t.
\end{equation}
We denote by $Q$ the homogeneous dimension of the group $(\RN, \circ, D_\lambda)$, and by $\delta_\lambda$ the family of dilations in $\R^{N+1}$ defined by $\delta_\lambda(x,t)=(D_\lambda(x), \lambda^2 t)$. It is well-known that $\H_0$ has a global fundamental solution $\Gamma(\cdot, \cdot)$ which is left-invariant and $\delta_\lambda$-homogeneous of degree $-Q$, i.e. $\Gamma((x,t),(\xi,\tau))=\Gamma((\xi^{-1}\circ x,t-\tau), 0)$ and $\Gamma(\delta_\lambda(z),\delta_\lambda(\zeta))=\lambda^{-Q}\Gamma(z,\zeta)$. Moreover, the Gaussian bounds \eqref{bounds} hold for $\Gamma$ with respect to a distance $d(\cdot,\cdot)$ which is left-invariant and $D_\lambda$-homogeneous of degree $1$ (we think such a distance $d$ as fixed in what follows). In particular we have $|B(x_0,r)|=r^Q|B(0,1)|$. We are going to show that a boundary point $z_0=(x_0,t_0)$ of a bounded open set $\Omega\subset\R^{N+1}$ is $\H_0$-regular if the complementary set of $\Omega$ contains the region
$$\left\{(x,t)\in\R^{N+1}\,:\, d^2(x,x_0)\geq C(t_0-t)\log\log\left(\frac{1}{t_0-t}\right),\mbox{ for }t\in \left(t_0-\min\{r_0^2,e^{-1}\},t_0\right)\right\}$$
for some $r_0>0$ and for some small enough positive constant $C$ (small enough in dependence of $Q$ and $b_0$). 
\begin{remark}
Both the presence of the $(\log\log)$-term and the presence of a restriction for the constant $C$ are known to be optimal in the following sense: if the set $\Omega$ is described around its boundary point $(x_0,t_0)$ by $\{|x-x_0|^2<C(t_0-t)\log\log(t_0-t)^{-1}\}$ for some constant $C>\frac{1}{b}>0$, then $(x_0,t_0)$ is irregular for the classical heat operator $\frac{1}{4b}\Delta-\de_t$. For this fact we refer the reader to the discussions in \cite[Section 7]{EK}, as well as to the classical counterexamples by Petrowski in \cite{Pe}. 
\end{remark}
With the following corollary we do not claim to determine the optimal range for $C$, but we do detect the sharp $(\log\log)$-behavior by exploiting the regularity criterion in Corollary \ref{infondobis}. As a matter of fact, we are going to bound from below the series in \eqref{suffcondmeasure} with the divergent series
$$\sum_k \frac{1}{k\log{k}}.$$
It will be clear with the proof that the terms $k\log{k}$ appear exactly because of their role in the definition of $\Omega^c_k(z_0)$ (as the sequence $\alpha(k)$ in \eqref{defalfa}).

\begin{corollary}\label{infondotris}
Let $\H_0$ be as in \eqref{homoper}, and let $d$ be the left-invariant homogeneous distance fixed above. Consider a bounded open set $\Omega$ in $\R^{N+1}$, and $z_0 \in \de\Omega$. There exists a positive constant $C^*=C^*(b_0,Q)$ such that, if we have 
$$\left\{(x,t)\in\R^{N+1}\,:\, d^2(x,x_0)\geq C(t_0-t)\log\log\left(\frac{1}{t_0-t}\right),\mbox{ for }t\in \left(t_0-\min\{r_0^2,e^{-1}\},t_0\right)\right\}\subset\R^{N+1}\smallsetminus\Omega$$
for some $r_0>0$ and $0<C< C^*$, then the point $z_0$ is $\H_0$-regular for $\de\Omega$.
\end{corollary}
\begin{proof}
We shall prove the statement with $C^*=\frac{1}{b_0}\frac{Q}{Q+4}$, where $Q$ is the homogeneous dimension of $(\RN, \circ, D_\lambda)$ and $b_0$ is the positive exponent in the Gaussian lower bound for $\Gamma$. By translation invariance, we can assume without loss of generality that $x_0=0$, that is $z_0=(0,t_0)$. Thus, for any $x\in\RN$ and $t<t_0$ we have
\begin{equation}\label{homogamma}
\Gamma(z_0,z)=\Gamma(0, (x,t-t_0))=\frac{1}{(t_0-t)^{\frac{Q}{2}}}\Gamma\left(0,\left(D_{\frac{1}{\sqrt{t_0-t}}}(x),-1\right)\right),\qquad\mbox{and}
\end{equation}
\begin{equation}\label{homogauss}
\frac{(t_0-t)^{-\frac{Q}{2}}}{\Lambda|B(0,1)|} e^{-b_0 \frac{d^2(x,0)}{t-t_0}}\leq \Gamma(0, (x,t_0-t))\leq \frac{\Lambda (t_0-t)^{-\frac{Q}{2}}}{|B(0,1)|} e^{-a_0 \frac{d^2(x,0)}{t_0-t}}.
\end{equation}
Fix $\Omega, C, r_0$ as in the assumptions. We also pick $\lambda\in(0,1)$, and we recall our notation $\alpha(k)=k\log k$. We claim the existence of $\rho>1$ and $k_1\in\N$ such that
\begin{eqnarray}\label{tobecontained}
\hspace{1cm}\Omega^c_k(z_0)\supseteq E_k:=\left\{(x,t)\in\R^{N+1}\,:\,\l^{-\alpha(k)}\leq\Gamma(0, (x,t-t_0)) \leq\l^{-\alpha(k+1)},\,\,\right. &&\\ 
\left.d^2(x,0)\geq \frac{Q}{Q+2}
\frac{t_0-t}{\rho^2 b_0}\log\log\left(\lambda^{-\frac{4}{Q}\alpha(k+1)}\right)\right\}&&\mbox{ for all }k\geq k_1.\nonumber
\end{eqnarray}
Let us first complete the proof of the desired statement by giving this claim for granted. We stress that, by \eqref{homogamma} and the homogeneity of $d$, we can write the set $E_k$ as
$$\left\{\left(D_{\sqrt{t_0-t}}(\xi),t\right)\in\R^{N+1} : \l^{\frac{2}{Q}\alpha(k+1)}\Gamma^{\frac{2}{Q}}\left(0,(\xi,-1)\right)\leq t_0-t\leq \l^{\frac{2}{Q}\alpha(k)}\Gamma^{\frac{2}{Q}}\left(0,(\xi,-1)\right),\,\,d(\xi,0)\geq
\frac{R_k}{\rho} \right\}$$
where $R^2_k=\frac{Q}{(Q+2)b_0}\log\log\left(\lambda^{-\frac{4}{Q}\alpha(k+1)}\right)$. Hence, by performing the change of variables $(x,t)\mapsto(\xi,t)$ with $\xi= D_{\frac{1}{\sqrt{t_0-t}}}(x)$, for every $k\geq k_1$ we deduce from \eqref{tobecontained} that
\begin{eqnarray*}
|\Omega^c_k(z_0)|&\geq& \int_{\left\{d(\xi,0)\geq
\frac{R_k}{\rho}\right\}}{\int_{t_0-\l^{\frac{2}{Q}\alpha(k)}\Gamma^{\frac{2}{Q}}\left(0,(\xi,-1)\right)}^{t_0-\l^{\frac{2}{Q}\alpha(k+1)}\Gamma^{\frac{2}{Q}}\left(0,(\xi,-1)\right)}{(t_0-t)^{\frac{Q}{2}}\,{\rm d}t}\,{\rm d}\xi}\\
&=&\frac{2 \l^{\frac{Q+2}{Q}\alpha(k)}}{ Q+2}\left(1-\l^{\frac{Q+2}{Q}(\alpha(k+1)-\alpha(k))}\right)\int_{\left\{\xi\,:\,d(\xi,0)\geq
\frac{R_k}{\rho}\right\}}{ \Gamma^{\frac{Q+2}{Q}}\left(0,(\xi,-1)\right) \,{\rm d}\xi}\\
&\geq& \l^{\frac{Q+2}{Q}\alpha(k)}\frac{2 \left(1- \l^{\frac{Q+2}{Q}\log 4}\right)}{ (Q+2) \left(\Lambda |B(0,1)|\right)^{\frac{Q+2}{Q}}}\int_{\left\{\xi\,:\,d(\xi,0)\geq \frac{R_k}{\rho}\right\}}{ e^{-b_0\frac{Q+2}{Q}d^2(\xi,0)} \,{\rm d}\xi},
\end{eqnarray*}
where in the last inequality we used that $\alpha(k+1)-\alpha(k)\geq \alpha(2)=\log 4$ and the lower bound in \eqref{homogauss}. We now notice, since $|B(0,\rho)|=|B(0,1)|\rho^Q$ and $Q\geq 1$, that
\begin{eqnarray*}
&&\int_{\left\{\xi\,:\,d(\xi,0)\geq \frac{R_k}{\rho}\right\}}{ e^{-b_0\frac{Q+2}{Q}d^2(\xi,0)} \,{\rm d}\xi}=\sum_{j=0}^{\infty}\int_{\left\{R_k\rho^{j-1}\leq d(\xi,0)\leq R_k\rho^j\right\}}{ e^{-b_0\frac{Q+2}{Q}d^2(\xi,0)} \,{\rm d}\xi}\\
&&\geq |B(0,1)|\sum_{j=0}^{\infty} e^{-b_0\frac{Q+2}{Q}\rho^{2j}R_k^2}R_k^Q\rho^{jQ}(1-\rho^{-Q})\geq \frac{|B(0,1)|}{\rho^Q}(1-\rho^{-Q})\sum_{j=0}^{\infty}\int_{R_k\rho^{j}}^{R_k\rho^{j+1}}{e^{-b_0\frac{Q+2}{Q}r^2}r^Q\,{\rm d}r}\\
&&=|B(0,1)|(\rho^{-Q}-\rho^{-2Q})\int_{R_k}^{+\infty}{e^{-b_0\frac{Q+2}{Q}r^2}r^Q\,{\rm d}r}\geq|B(0,1)|(\rho^{-Q}-\rho^{-2Q})\int_{R_k}^{+\infty}{re^{-b_0\frac{Q+2}{Q}r^2}\,{\rm d}r}\\
&&=\frac{Q|B(0,1)|(\rho^{-Q}-\rho^{-2Q})}{2b_0(Q+2)}e^{-b_0\frac{Q+2}{Q}R_k^2}.
\end{eqnarray*}
Therefore, if we put together the last two estimates and we substitute the value of $R_k$, we establish the existence of a positive constant $c_0$ such that
$$
\frac{|\Omega^c_k(z_0)|}{\l^{\frac{Q+2}{Q}\alpha(k)}}\geq c_0 e^{-b_0\frac{Q+2}{Q}R_k^2} = c_0 e^{- \log\log\left(\lambda^{-\frac{4}{Q}\alpha(k+1)}\right)}=\frac{c_0}{\log\left(\lambda^{-\frac{4}{Q}}\right)}\frac{1}{\alpha(k+1)}\quad\forall\, k\geq k_1.
$$
Thus, the series in \eqref{suffcondmeasure} can be estimated from below with the series
$$\frac{c_0}{\log\left(\lambda^{-\frac{4}{Q}}\right)}\sum_{k=k_1}^\infty \frac{1}{\alpha(k+1)}$$
which is divergent since $\alpha(k+1)=(k+1)\log(k+1)$. Corollary \ref{infondobis} yields then the $\H_0$-regularity of the point $z_0$.\\
We are now left with the proof of the claim \eqref{tobecontained}. Recalling the definition \eqref{defOmk} of $\Omega^c_k(z_0)$,
this is the same as showing that there exist $\rho>1$ and $k_1\in\N$ such that $E_k\subseteq \R^{N+1}\smallsetminus\O$ for every $k\geq k_1$.
Then, by the main assumption on the complementary set of $\O$, it is enough to show that 
\begin{equation}\label{toshowinside}
E_k\subseteq  \left\{(x,t)\in\R^{N+1}\,:\, d^2(x,0)\geq C(t_0-t)\log\log\left(\frac{1}{t_0-t}\right)\right\}\cap \left(\RN\times\left(t_0-\min\{r_0^2,e^{-1}\},t_0\right)\right)
\end{equation}
for all $k\geq k_1$. To see this, we keep in mind that $\frac{1}{C}>\frac{1}{C^*}=b_0\frac{Q+4}{Q}=b_0\frac{2}{Q}+b_0\frac{Q+2}{Q}$, and we fix $\rho>1$ through the relation
$$\frac{1}{C}=b_0\frac{2}{Q}+\rho^2b_0\frac{Q+2}{Q}.$$
This implies in particular, using that $e^\sigma\geq 1+\sigma$ for all $\sigma$, that
\begin{equation}\label{primadafare}
e^{\frac{d^2(x,0)}{C (t_0-t)}}\geq e^{\rho^2b_0\frac{Q+2}{Q}\frac{d^2(x,0)}{(t_0-t)}}+b_0\frac{2}{Q}\frac{d^2(x,0)}{t_0-t}\quad\forall\, t<t_0\mbox{ and }x\in\RN.
\end{equation}
Let us also fix $k_1\in\N$ such that
\begin{equation}\label{fissokappuno}
\lambda^{\alpha(k_1+1)}\leq \frac{1}{\Lambda|B(0,1)|}\qquad\mbox{ and }\qquad \lambda^{\alpha(k_1)}< \frac{|B(0,1)|}{\Lambda}\min\left\{r_0^Q, e^{-\frac{Q}{2}}\right\}.
\end{equation}
The first inequality in \eqref{fissokappuno} ensures that
\begin{equation}\label{secondafare}
\log\left(\frac{1}{\lambda^{\frac{4}{Q}\alpha(k+1)}}\right)\geq \frac{2}{Q} \log\left(\frac{\Lambda|B(0,1)|}{\lambda^{\alpha(k+1)}}\right)\quad\forall\, k\geq k_1.
\end{equation}
Moreover, if $z=(x,t)\in E_k$, we have $\Gamma(0, (x,t-t_0))\leq \lambda^{-\alpha(k+1)}$ which implies by \eqref{homogauss}
\begin{equation}\label{terzadafare}
\frac{2}{Q}\log\left(\frac{\Lambda|B(0,1)|e^{b_0\frac{d^2(x,0)}{t_0-t}}}{\lambda^{\alpha(k+1)}}\right)\geq\log\left(\frac{1}{t_0-t}\right).
\end{equation}
If we combine \eqref{primadafare}, \eqref{secondafare} and \eqref{terzadafare}, for all $k\geq k_1$ and for any $z\in E_k$ we get
$$
e^{\frac{d^2(x,0)}{C (t_0-t)}}\geq \log\left(\frac{1}{\lambda^{\frac{4}{Q}\alpha(k+1)}}\right)+b_0\frac{2}{Q}\frac{d^2(x,0)}{t_0-t}\geq \frac{2}{Q}\log\left(\frac{\Lambda|B(0,1)|}{\lambda^{\alpha(k+1)}}\right)+\frac{2}{Q}\log\left(e^{b_0\frac{d^2(x,0)}{t_0-t}}\right)\geq \log\left(\frac{1}{t_0-t}\right),
$$
which says
$$d^2(x,0)\geq C(t_0-t)\log\log\left(\frac{1}{t_0-t}\right)$$
as desired. On the other hand, if $z\in E_k$ we know by \eqref{homogauss} that $\lambda^{-\alpha(k)}\leq \frac{\Lambda}{|B(0,1)| (t_0-t)^{\frac{Q}{2}}}$, and from the second inequality in \eqref{fissokappuno} we then obtain
$$t_0-t\leq \left(\frac{\Lambda \lambda^{\alpha(k)}}{|B(0,1)|}\right)^{\frac{2}{Q}}< \min\{r_0^2,e^{-1}\}\quad\forall\,k\geq k_1.$$
This completes the proof of \eqref{toshowinside}, and the proof of the corollary.
\end{proof}
To fix the ideas, we can say that $(x_0,t_0)$ is $\H_0$-regular for $\de\Omega$ if $\O$ is given by the set
$$\left\{(x,t)\in\R^{N+1}\,:\, d^2(x,x_0)< \frac{Q}{Q+5}\frac{t_0-t}{b_0}\log\log\left(\frac{1}{t_0-t}\right),\, t\in \left(t_0-\frac{1}{2e},t_0\right)\right\}.$$
The geometric condition for the regularity in Corollary \ref{infondotris} appears to be new for the whole class of homogeneous operators $\H_0$ in \eqref{homoper} (except for the classical heat equation in Euclidean $\RN$).
As a straightforward consequence, we can infer that a boundary point $(x_0,t_0)$ of a bounded open set $\Omega$ is $\H_0$-regular if there exist  $M, r_0>0$ such that
$$\left\{(x,t)\in\R^{N+1}\,:\, d^{2}(x,x_0)\geq M(t_0-t),\mbox{ for }t\in (t_0-r_0^2,t_0)\right\}\subset\R^{N+1}\smallsetminus\Omega.$$

\section*{Acknowledgments}

\noindent The authors wish to thank Ermanno Lanconelli for having suggested the problem, and the anonymous referee whose valuable comments led to an improvement of the manuscript. G.T. has been partially supported by the Gruppo Nazionale per l'Analisi Matematica, la Probabilit\`a e le loro Applicazioni (GNAMPA) of the Istituto Nazionale di Alta Matematica (INdAM).

\bibliographystyle{amsplain}

\begin{thebibliography}{10}

\bibitem{AKN}
B. Avelin, T. Kuusi, K. Nystr\"om,
\textit{Boundary behavior of solutions to the parabolic p-Laplace equation}.
Anal. PDE 12 (2019) 1--42

\bibitem{BBG} 
A. Bj\"orn, J. Bj\"orn, U. Gianazza,
\textit{The Petrovski{\u\i} criterion and barriers for degenerate and singular $p$-parabolic equations}.
Math. Ann. 368 (2017) 885--904

\bibitem{BUt}
A. Bonfiglioli, F. Uguzzoni,
\textit{Harnack inequality for non-divergence form operators on stratified groups}.
 Trans. Amer. Math. Soc. 359 (2007) 2463--2482

\bibitem{BBLU}
M. Bramanti, L. Brandolini, E. Lanconelli, F. Uguzzoni,
\textit{Non-divergence equations structured on H\"{o}rmander vector fields: heat kernels and Harnack inequalities}.
 Mem. Amer. Math. Soc. 204 no. 961 (2010)

\bibitem{EK}
E.G. Effros, J.L. Kazdan,
\textit{Applications of Choquet simplexes to elliptic and parabolic boundary value problems}.
 J. Differential Equations 8 (1970) 95--134

\bibitem{EG}
L.C. Evans, R.F. Gariepy,
\textit{Wiener's criterion for the heat equation}.
Arch. Rational Mech. Anal. 78 (1982) 293--314

\bibitem{FGL}
E.B. Fabes, N. Garofalo, E. Lanconelli,
\textit{Wiener's criterion for divergence form parabolic operators with $C^1$-Dini continuous coefficients}.
Duke Math. J. 59 (1989) 191--232

\bibitem{GL}
N. Garofalo, E. Lanconelli,
\textit{Wiener's criterion for parabolic equations with variable coefficients and its consequences}.
Trans. Amer. Math. Soc. 308 (1988) 811--836

\bibitem{GS}
N. Garofalo, F. Seg{\`a}la,
\textit{Estimates of the fundamental solution and Wiener's criterion for the heat equation on the Heisenberg group}.
Indiana Univ. Math. J. 39 (1990) 1155--1196

\bibitem{KL}
T. Kilpel\"ainen, P. Lindqvist,
\textit{On the Dirichlet boundary value problem for a degenerate parabolic equation}.
SIAM J. Math. Anal. 27 (1996) 661--683

\bibitem{KKKP} 
J. Kinnunen, R. Korte, T. Kuusi, M. Parviainen,
\textit{Nonlinear parabolic capacity and polar sets of superparabolic functions}.
Math. Ann. 355 (2013) 1349--1381

\bibitem{KLTade}
A.E. Kogoj, E. Lanconelli, G. Tralli,
\textit{An inverse mean value property for evolution equations}. 
Adv. Differential Equations 19 (2014), 783--804

\bibitem{KLT}
A.E. Kogoj, E. Lanconelli, G. Tralli,
\textit{Wiener-Landis criterion for Kolmogorov-type operators}. 
Discrete Contin. Dyn. Syst. 38 (2018), 2467--2485

\bibitem{L73}
E. Lanconelli,
\textit{Sul problema di Dirichlet per l'equazione del calore}.
Ann. Mat. Pura Appl. (4) 97 (1973) 83--114

\bibitem{L77}
E. Lanconelli,
\textit{Sul confronto della regolarit{\`a} dei punti di frontiera rispetto ad operatori lineari parabolici diversi}.
Ann. Mat. Pura Appl. (4) 114 (1977) 207--227

\bibitem{LP99}
E. Lanconelli, A. Pascucci,
\textit{Superparabolic functions related to second order hypoelliptic operators}.
Potential Anal. 11 (1999) 303--322

\bibitem{LTU}
E. Lanconelli, G. Tralli, F. Uguzzoni,
\textit{Wiener-type tests from a two-sided Gaussian bound}. 
Ann. Mat. Pura Appl. (4) 196 (2017) 217--244

\bibitem{LU}
E. Lanconelli, F. Uguzzoni,
\textit{Potential analysis for a class of diffusion equations: a Gaussian bounds approach}.
 J. Differential Equations 248 (2010) 2329--2367

\bibitem{La}
E.M. Landis,
\textit{Necessary and sufficient conditions for the regularity of a boundary point for the Dirichlet problem for the heat equation}.
Dokl. Akad. Nauk SSSR 185 (1969) 517--520

\bibitem{Nov}
A.A. Novruzov,
\textit{Certain criteria for the regularity of boundary points for linear and quasilinear parabolic equations}.
Dokl. Akad. Nauk SSSR 209 (1973) 785--787

\bibitem{Pe}
I. Petrowsky,
\textit{Zur ersten Randwertaufgabe der W{\"a}rmeleitungsgleichung}.
Compositio Math. 1 (1935) 383--419

\bibitem{Pi}
B. Pini,
\textit{Sulla soluzione generalizzata di Wiener per il primo problema di valori al contorno nel caso parabolico}.
Rend. Sem. Mat. Univ. Padova 23 (1954) 422--434

\bibitem{Ro}
K.L. Rotz,
\textit{Monotonicity Formulas for Diffusion Operators on Manifolds and Carnot Groups, Heat Kernel Asymptotics and Wiener's Criterion on Heisenberg-type Groups}.
Ph.D. thesis, Purdue University, 2016.

\bibitem{Sc}
V. Scornazzani,
\textit{The Dirichlet problem for the Kolmogorov operator}.
 Boll. Un. Mat. Ital. C (5) 18 (1981) 43--62

\bibitem{TUjd}
G. Tralli, F. Uguzzoni,
\textit{Wiener criterion for X-elliptic operators}. 
J. Differential Equations 259 (2015) 6510--6527

\bibitem{TUcv}
G. Tralli, F. Uguzzoni,
\textit{On a non-smooth potential analysis for H\"ormander-type operators}.
Calc. Var. Partial Differential Equations (2018) 57: 37, DOI: \verb|10.1007/s00526-018-1301-6|

\bibitem{Ug}
F. Uguzzoni,
\textit{Estimates of the Green function for X-elliptic operators}.
 Math. Ann. 361 (2015) 169--190

\end{thebibliography}

\end{document}